\documentclass[11pt, reqno]{amsart}

\usepackage[utf8]{inputenc}
\usepackage[T1]{fontenc}
\usepackage[english]{babel}
\usepackage{amsmath, amsthm}
\usepackage{ae}
\usepackage{icomma}
\usepackage{units}
\usepackage{color}
\usepackage{graphicx}
\usepackage{bbm}
\usepackage{caption}
\usepackage{array}
\usepackage[hmarginratio=1:1]{geometry}
\usepackage[hyphens]{url}
\usepackage[pdfpagelabels=false, hidelinks]{hyperref}
\usepackage{mathrsfs}
\usepackage{amssymb, amsfonts}
\usepackage{placeins} 
\usepackage{tikz}
\usetikzlibrary{patterns}
\usepackage{float} 
\usepackage[nodayofweek]{datetime}
\usepackage{enumitem}
\usepackage{mathtools}
\usepackage[numbers, sort]{natbib}
\usepackage{dsfont}
\usepackage{comment}

\newcommand{\R}{\ensuremath{\mathbb{R}}}

\renewcommand{\S}{\ensuremath{\mathbb{S}}}
\newcommand{\Haus}{\ensuremath{\mathcal{H}}}
\newcommand{\1}{\ensuremath{\mathds{1}}}

\newcommand{\eps}{\ensuremath{\varepsilon}}

\DeclareMathOperator{\dist}{\textnormal{dist}}

\DeclareMathOperator{\Tr}{Tr}
\DeclareMathOperator{\id}{\mathds{1}}
\DeclareMathOperator{\supp}{supp}

\newtheorem{theorem}{Theorem}[section]

\newtheorem{lemma}[theorem]{Lemma}
\newtheorem{proposition}[theorem]{Proposition}
\newtheorem{corollary}[theorem]{Corollary}

\numberwithin{theorem}{section}
\theoremstyle{remark}
\newtheorem{remark}[theorem]{Remark}

\newcommand{\limplus}{{\mathchoice{\vcenter{\hbox{$\scriptstyle +$}}}
  {\vcenter{\hbox{$\scriptstyle +$}}}
  {\vcenter{\hbox{$\scriptscriptstyle +$}}}
  {\vcenter{\hbox{$\scriptscriptstyle +$}}}
}}

\begin{document}

\title[The normal derivative of the Lane--Emden ground state]{An inequality for the normal derivative\\ of the Lane--Emden ground state}

\author[R. L. Frank]{Rupert L. Frank}
\address{\textnormal{(R. L. Frank)} Mathematisches Institut, Ludwig-Maximilians Universit\"at M\"unchen, Theresienstr. 39, 80333 M\"unchen, Germany, Munich Center for Quantum Science and Technology, Schellingstr. 4, 80799 M\"unchen, Germany, and Department of Mathematics, California Institute of Technology, Pasadena, CA 91125, USA}
\email{r.frank@lmu.de}

\author[S. Larson]{Simon Larson}
\address{\textnormal{(S. Larson)} Mathematical Sciences, Chalmers University of Technology and the University of Gothenburg, SE-41296 Gothenburg, Sweden}
\email{larsons@chalmers.se}

\subjclass[2020]{}
\keywords{}

\begin{abstract}
	We consider Lane--Emden ground states with polytropic index $0\leq q-1\leq 1$, that is, minimizers of the Dirichlet integral among $L^q$-normalized functions. Our main result is a sharp lower bound on the $L^2$-norm of the normal derivative in terms of the energy, which implies a corresponding isoperimetric inequality. Our bound holds for arbitrary bounded open Lipschitz sets $\Omega\subset\R^d$, without assuming convexity.
\end{abstract}

\thanks{\copyright\, 2022 by the authors. This paper may be
reproduced, in its entirety, for non-commercial purposes.\\
Partial support through U.S. National Science Foundation grant DMS-1954995 (R.L.F.), through the Deutsche For\-schungs\-gemeinschaft (DFG, German Research Foundation) through Germany’s Excellence Strategy EXC-2111-390814868 (R.L.F.), and through Knut and Alice Wallenberg Foundation grant KAW~2021.0193 (S.L.) is acknowledged. The authors would like to thank the anonymous referee for helpful suggestions.}

\maketitle

\section{Introduction and main results}

We are interested in sharp lower bounds on the normal derivative of minimizers of the variational problem
\begin{equation}\label{def: semi-linear eigenvalue}
	\lambda_q(\Omega) := \inf_{u \in C^\infty_0(\Omega)\setminus\{0\}} \frac{\|\nabla u\|_{L^2(\Omega)}^2}{\|u\|_{L^q(\Omega)}^2}\, .
\end{equation}
Here $1\leq q\leq 2$ is a parameter and $\Omega \subset \R^d$ is a non-empty open set. 

The most important cases are $q=2$, where $\lambda_2(\Omega)$ is the the bottom of the spectrum of the Dirichlet Laplacian in $\Omega$, and $q=1$, where $\lambda_1(\Omega)$ is the inverse torsional rigidity of $\Omega$. For general $1\leq q\leq 2$, the minimization problem $\lambda_q(\Omega)$ arises in connection with ground state solutions of the Lane--Emden equation with polytropic index $q-1$; see \eqref{eq: nonlinear equation} below.

Some of the results discussed in this paper, most importantly a Brunn--Minkowski type inequality for $\lambda_q$ (Theorem \ref{thm: BM lambdaq}), are valid for arbitrary open sets $\Omega$. Others require some modest assumptions, namely that the open set $\Omega$ is bounded and has Lipschitz boundary. Under these assumptions it is well known (see Section~\ref{sec: Prelminaries} for references) that there is a non-negative minimizer of~\eqref{def: semi-linear eigenvalue}. (Strictly speaking, it is a minimizer of the corresponding problem with $C^\infty_0(\Omega)$ replaced by $H^1_0(\Omega)$.) By homogeneity, we can choose a minimizer to be normalized in $L^q(\Omega)$. If either $1\leq q<2$ or if $q=2$ and $\Omega$ is connected, then the non-negativity and the normalization determine the minimizer uniquely. If $q=2$ and $\Omega$ has multiple connected components, then there may be several such minimizers. In this case, all our statements are valid for any choice of minimizer. In what follows, we denote a non-negative and normalized minimizer by $u_{q, \Omega}$.

Our main result is an isoperimetric-type inequality for the $L^2$-norm of the normal derivative of the minimizer $u_{q, \Omega}$ at the boundary. In what sense the normal derivative should be understood when the boundary is irregular is discussed in Section~\ref{sec: normal derivative}. Specifically, the result is the following theorem.

\begin{theorem}\label{thm: semi-linear eigenvalue}
	Fix $1\leq q \leq 2$, let $\Omega \subset \R^d$ be open and bounded with Lipschitz boundary, and define $\alpha_q :=(2+d(2/q-1))^{-1}$. Then
	\begin{equation}
		\label{eq:mainineq}
		\int_{\partial\Omega} \Bigl(\frac{\partial u_{q, \Omega}}{\partial \nu}\Bigr)^2\, d\Haus^{d-1}(x) \geq 
		\frac{\lambda_q(\Omega)^{1+\alpha_q}}{\alpha_q\, \lambda_q(B)^{\alpha_q}}\, , 
	\end{equation}
	where $u_{q, \Omega}$ is an $L^q$-normalized minimizer associated to $\lambda_q(\Omega)$, $\frac{\partial u}{\partial \nu}$ is the derivative of $u$ in the direction of the outward normal to $\partial \Omega$,  and $B$ is the unit ball. Moreover, equality holds if $\Omega$ is a ball.
\end{theorem}

In~\eqref{eq:mainineq}, $d\mathcal H^{d-1}$ denotes integration with respect to $(d-1)$-dimensional Hausdorff measure, which is simply the surface measure on $\partial\Omega$.

In the linear case, i.e.\ $q=2$, the inequality simplifies to 
\begin{equation*}
	\int_{\partial\Omega} \Bigl(\frac{\partial u_{2, \Omega}}{\partial n}\Bigr)^2\, d\Haus^{d-1}(x) \geq 2\, \frac{\lambda_2(\Omega)^{3/2}}{\lambda_2(B)^{1/2}}\, .
\end{equation*}
For $q=1$ the inequality can equivalently be written as
\begin{equation*}
	\int_{\partial\Omega} \Bigl(\frac{\partial v_\Omega}{\partial n}\Bigr)^2\, d\Haus^{d-1}(x)\geq (d+2)\, T(B)^{1/(d+2)}T(\Omega)^{1-1/(d+2)}\, , 
\end{equation*}
where $v_\Omega$ denotes the torsion function of $\Omega$, that is, the unique solution of
\begin{equation*}
	\begin{cases}
		-\Delta v_\Omega =1 \quad & \mbox{in }\Omega\, , \\
		v_\Omega =0 & \mbox{on }\partial\Omega\, , 
	\end{cases}
\end{equation*}
and $T(\Omega):= \int_\Omega |\nabla v_\Omega|^2\, dx -2 \int_\Omega v_\Omega\, dx = - \int_\Omega |\nabla v_\Omega|^2\, dx$ denotes the torsional rigidity.

If $\Omega \subset \R^d$ is an open set of finite measure and $\Omega^*$ denotes an open ball with the same measure, then a classical rearrangement argument (see, for instance,~\cite{LiebLoss}) implies the Faber--Krahn-type inequality $\lambda_q(\Omega)\geq \lambda_q(\Omega^*)$. When combined with Theorem~\ref{thm: semi-linear eigenvalue}, one obtains the following isoperimetric inequality.

\begin{corollary}\label{cor}
	Fix $1\leq q \leq 2$ and let $\Omega \subset \R^d$ be open and bounded with Lipschitz boundary. Then
\begin{equation*}
	\int_{\partial\Omega} \Bigl(\frac{\partial u_{q, \Omega}}{\partial \nu}\Bigr)^2\, d\Haus^{d-1}(x)
	\geq \int_{\partial\Omega^*} \Bigl( \frac{\partial u_{q, \Omega^*}}{\partial \nu}\Bigr)^2\, d\Haus^{d-1}(x)\, , 
\end{equation*}	
with equality if $\Omega$ is a ball.
\end{corollary}

As far as we know, Theorem~\ref{thm: semi-linear eigenvalue} and Corollary~\ref{cor} are new. In the special where $\Omega$ is convex and $q=2$, the inequality in Theorem~\ref{thm: semi-linear eigenvalue}, while not explicitly stated, can be deduced relatively easily from results proved in Jerison's work~\cite{Jerison_Adv96}. (Indeed, the corresponding inequality is written out in his analogous work on the capacity problem for convex sets; see~\cite[Corollary 3.19]{Jerison_Acta96}.) Similarly, still assuming that $\Omega$ is convex, the inequality in Theorem~\ref{thm: semi-linear eigenvalue} for $q=1$ could be deduced from the work of Colesanti and Fimiani in~\cite{ColesantiFimiani}. The inequality, for $\Omega$ convex and for $q\in\{1, 2\}$, appears explicitly in~\cite[Subsection 3.2]{BucurFragalaLamboley_ESAIM2012}; see also~\cite{Colesanti_etal_pCap_adv2015} for related results in the convex setting. Many of the above references are primarily concerned with generalizations of the Minkow\-ski problem to set functionals on convex sets. For the set functional $\lambda_q$ with $q=2$ this is treated in \cite{Jerison_Adv96} and for $q=1$ and $1<q<2$ we refer to \cite{ColesantiFimiani} and \cite{QiXi}, respectively. Thus, what we accomplish here is to extend the inequality to the full range $1\leq q\leq 2$ and, more importantly, to remove the convexity assumption.

The above works in the convex case use representation formulas for $\lambda_q(\Omega)$ in terms of an integral of the support function against certain measures on $\S^{d-1}$. These formulas appear prominently in the assertions and proofs in the convex case. They do not have an analogue in the non-convex case. However, as we show here, while these formulas are a convenient tool in the convex case, they are not essential for the proof of the inequality in Theorem~\ref{thm: semi-linear eigenvalue} and all relevant assertions can be proved without them. This leads to a number of significant new difficulties that we need to overcome.

Our proof of Theorem~\ref{thm: semi-linear eigenvalue} has two main ingredients, namely a Brunn--Minkowski inequality for $\lambda_q$ and the computation of the derivative of the function $t\mapsto \lambda_q(\Omega + tB)$ at $t=0$. In the remainder of this introduction, we discuss these two ingredients in some more detail and explain how they yield our theorem.

Before doing this, however, in order to motivate our arguments, let us recall how the classical isoperimetric inequality follows from the Brunn--Minkowski inequality. The latter inequality states that, for any non-empty compact sets $\Omega_0, \Omega_1 \subset \R^d$ and $0\leq t\leq 1$, 
\begin{equation}\label{eq: BM}
	|(1-t) \Omega_0+t \Omega_1|^{1/d} \geq (1-t)|\Omega_0|^{1/d}+t|\Omega_1|^{1/d}\, .
\end{equation}
Here and in what follows, $|\cdot|$ denotes the Lebesgue measure and, for $\Omega, \Omega'\subset  \R^d, s\geq 0$,  $s \Omega$ denote the dilation of $\Omega$ by $s$, and $+$ denotes the Minkowski sum, that is, 
\begin{align*}
	s \Omega := \{s x  : x \in \Omega\}\quad \mbox{and} \quad 
	\Omega + \Omega' := \{x+y: x\in \Omega, y\in \Omega'\}\, .
\end{align*} 
By setting $\Omega_1=B$, the unit ball in $\R^d$, and differentiating~\eqref{eq: BM} with respect to $t$ at $t=0$, for sufficiently regular sets one obtains the classical isoperimetric inequality. In fact, for arbitrary sets one arrives at an isoperimetric inequality not for the perimeter but for the so-called lower outer Minkowski content defined by
\begin{equation*}
	\mathcal{SM}_*(\Omega) := \liminf_{t \to 0^\limplus} \frac{|(\Omega+t B)\setminus \Omega|}{t }\, .
\end{equation*}
If $\Omega$ is sufficiently regular, for instance, if $\partial\Omega$ is Lipschitz, then the lower outer Minkowski content agrees with the perimeter of $\Omega$ and one arrives at the classical isoperimetric inequality~\cite{Hadwiger}. However, under what geometric assumptions $\mathcal{SM}_*$ and perimeter agree is not a trivial question.

Here we are interested in inequalities that arise by mimicking this argument for $\lambda_q$ instead of Lebesgue measure. As we shall see, these set functions also satisfy Brunn--Minkowski-type inequalities. The main issues that we need to deal with when carrying out this procedure are similar to those alluded to above; namely, when can the differentiation be justified, and when does the derivative agree with the quantity we aim to bound. 

\subsection{Outline of proof}

We begin by discussing the first ingredient of our proof, namely a Brunn--Minkowski inequality for $\lambda_q$. This inequality goes back to Brascamp--Lieb~\cite{BrascampLieb_JFA76} in the case $q=2$, Borell~\cite{Borell_85} for $q=1$, and Colesanti~\cite{colesanti_brunnminkowski_2005} for $1\leq q<2$. However, in these works, the inequality is stated under unnecessarily restrictive assumptions on $\Omega$. In Section~\ref{sec: BM inequality}, we use a simple approximation argument to show that the Brunn--Minkowski inequality remains valid under weaker assumptions on the geometry. Moreover, in the appendix we show that Colesanti's characterization of the cases of equality for $1\leq q<2$ is valid without any further assumptions.

Specifically, we deduce the following result.

\begin{theorem}\label{thm: BM lambdaq}
	Fix $1\leq q \leq 2$, let $\Omega_0, \Omega_1\subset \R^d$ be non-empty open sets, and set $\alpha_q :=(2+d(2/q-1))^{-1}$. Then, for all $0< t < 1$, 
	\begin{equation}\label{eq: BM ineq thm}
		\lambda_{q}((1-t)\Omega_0+t\Omega_1) \leq \Bigl((1-t)\lambda_{q}(\Omega_0)^{-\alpha_q}+t\lambda_{q}(\Omega_1)^{-\alpha_q}\Bigr)^{-1/\alpha_q} \, , 
	\end{equation}
	here the right-hand side should be understood as zero if $\min\{\lambda_q(\Omega_0), \lambda_q(\Omega_1)\}=0$.

	If $1\leq q<2$, $\min\{\lambda_q(\Omega_0), \lambda_q(\Omega_1)\}>0$, and equality holds in~\eqref{eq: BM ineq thm} for some $t\in (0, 1)$, then there is an open bounded convex set $K\subset \R^d$ such that $\Omega_0$ and $\Omega_1$ agree with homothetic images of $K$ up to sets of capacity zero. 
\end{theorem}

\begin{remark}
	Without additional geometric assumptions the characterization of equality does not extend to the case $q=2$; this follows by considering sets with multiple connected components, for instance $\Omega_0 = (0, 1)^d \cup \bigl( (2, 3)\times (0, 1)^{d-1} \bigr)$ and $\Omega_1=(0, 1)^d$.
\end{remark}

Our second ingredient concerns the derivative of $t\mapsto \lambda_q(\Omega + tB)$ at $t=0$. This is closely related to a Hadamard variation formula. In general, if $\Lambda$ is a real-valued set function, then one may ask how $\Lambda$ behaves with respect to perturbations around a given set $\Omega$. The most common manner in which to analyze such questions is to compute the Fr\'echet derivative of the map $\Phi \mapsto \Lambda(\Phi(\Omega))$ at $\Phi=\1$, where $\Phi$ is a diffeomorphism in a neighborhood of $\Omega$. Formulas of this type are typically called Hadamard formulas.

However, the geometric perturbations $\Omega +tB$ that appear in the Brunn--Minkowski formula in Theorem~\ref{thm: BM lambdaq} can not in general be parametrized by a family of local diffeomorphisms. As such, our desired result lies somewhat outside the standard theory. Our approach will be to first prove a Hadamard formula for $\lambda_q$ (see Section~\ref{sec: Hadamard}) and then show that, for $\Omega$ with $C^1$ boundary, the curve $\Omega+ tB$ can be approximated well enough by diffeomorphisms of $\Omega$ to compute the desired derivative (see Section~\ref{sec: Geometry}). Precisely, we prove the following theorem.

\begin{theorem}\label{thm: Hadamard Minkowski sum}
  Fix $1\leq q\leq 2$ and let $\Omega \subset \R^d$ be open, bounded, and connected with $C^1$ boundary. Then
  \begin{equation*}
   \lim_{t \to 0^\limplus} \frac{\lambda_q(\Omega + tB)-\lambda_q(\Omega)}{t} = -\int_{\partial \Omega} \Bigl( \frac{\partial u_{q, \Omega}}{\partial \nu}\Bigr)^2\, d\Haus^{d-1}(x)\, , 
 \end{equation*}
 where $u_{q, \Omega}$ is an $L^q$-normalized minimizer associated to $\lambda_q(\Omega)$, $\frac{\partial u}{\partial \nu}$ is the derivative of $u$ in the direction of the outward normal to $\partial \Omega$,  and $B$ is the unit ball.
\end{theorem}

A related Hadamard formula was proved in \cite{HuSoXu} in the case where $\Omega$ is convex with $C^2$ boundary having positive Gauss curvature. The proof in \cite{HuSoXu} depends on the convexity but, as we show here, the validity of the result does not.

By combining Theorems~\ref{thm: BM lambdaq} and~\ref{thm: Hadamard Minkowski sum}, one readily deduces the inequality in Theorem~\ref{thm: semi-linear eigenvalue} under the assumption that $\Omega$ has $C^1$ boundary. To obtain the result under the weaker assumption of Lipschitz boundary, we need some rather deep results on the existence of the normal derivative due to Jerison and Kenig~\cite{JerisonKenig_BAMS81}; see also~\cite{verchota_layer_1984}.

This concludes our sketch of the strategy of the proof of Theorem~\ref{thm: semi-linear eigenvalue}. We end this introduction by emphasizing that for $q=2$ our bounds concern the lowest eigenvalue of the Laplacian. For bounds for higher eigenvalues, including those of variable coefficient operators, see, e.g.,~\cite{HaTa} and~\cite[Theorem 4.4]{KeLiSh}. These bounds, however, do not have an isoperimetric character. For an isoperimetric upper bound on $\| u_{q, \Omega} \|_{L^k(\Omega)}$ with $k\geq q$, see~\cite{HuDa}.


\section{Preliminaries}\label{sec: Prelminaries}

\subsection{The minimization problem \texorpdfstring{$\lambda_q(\Omega)$}{lambda} and its minimizer}

In this subsection, we discuss some aspects of the minimization problem~\eqref{def: semi-linear eigenvalue} that we shall need later on. Throughout we shall assume that $\Omega \subset \R^d$, $d\geq 2$, is a non-empty open set. 

Clearly $\lambda_q$ is invariant under translations, while under dilations it obeys
\begin{equation*}
	\lambda_q(s\Omega)= s^{-1/\alpha_q}\lambda_q(\Omega)\quad \text{for all}\ s>0
\end{equation*}
with $\alpha_q = (2+d(2/q-1))^{-1}$.

If $\Omega = \cup_{j\geq 1} \Omega_j$, $\Omega_j \cap \Omega_{j'}=\emptyset$ for $j \neq j'$, then the quantity $\lambda_q(\Omega)$ can be written in terms of the corresponding quantities for the elements of the union, namely, 
\begin{equation}\label{eq: spin formula0}
	\lambda_q(\Omega) = \biggl(\sum_{j\geq 1}\lambda_{q}(\Omega_j)^{-\frac{q}{2-q}}\biggr)^{- \frac{2-q}{q}} \ \text{if}\ 1\leq q<2 \, , \quad \mbox{and}\quad 
	\lambda_2(\Omega) = \min_{j\geq 1} \lambda_2(\Omega_j) \, ;
\end{equation}
see, e.g.,~\cite{BrascoFranzina_Semilinear_overview_2020}. We remark that for $q=1$ the first formula is nothing but the additivity of the torsional rigidity under disjoint unions.

We record a simple continuity property.

\begin{lemma}\label{lem: Hausdorff interior continuity}
	Fix $1\leq q \leq 2$ and let $\Omega \subset \R^d$ be a non-empty open set. Let $\{\Omega_j\}_{j\geq 1}$ be a sequence of open sets with $\Omega_j \subset \Omega_{j+1}$ for all $j\geq 1$, $\cup_{j\geq 1}\Omega_j= \Omega$, and such that $\Omega_j \cap B_R$ converges to $\Omega \cap B_R$ with respect to the Hausdorff distance for any $R>0$. Then
	\begin{equation*}
		\lim_{j \to \infty} \lambda_q(\Omega_j) = \lambda_q(\Omega)\, .
	\end{equation*}
\end{lemma}

\begin{proof}
	By monotonicity under set inclusions $\lambda_q(\Omega)\leq \lambda_q(\Omega_j)$ for all $j$. To prove the reverse inequality we argue as follows. For any $\eps>0$ there exists $\varphi \in C_0^\infty(\Omega)$ such that
	\begin{equation*}
		\lambda_q(\Omega)\geq \frac{\|\nabla \varphi\|_{L^2(\Omega)}^2}{\|\varphi\|_{L^q(\Omega)}^2}-\eps\, .
	\end{equation*}
	Since $\supp \varphi$ is compact and $\dist(\supp \varphi, \partial\Omega)>0$, it holds for $j$ sufficiently large that $\supp \varphi \subset \Omega_j$. Consequently, by the definition of $\lambda_q(\Omega_j)$, 
	\begin{equation*}
		\frac{\|\nabla \varphi\|_{L^2(\Omega)}^2}{\|\varphi\|_{L^q(\Omega)}^2}
		=
		\frac{\|\nabla \varphi\|_{L^2(\Omega_j)}^2}{\|\varphi\|_{L^q(\Omega_j)}^2}\geq \lambda_q(\Omega_j)\, .
	\end{equation*}
	Thus, for any $\eps>0$ and $j$ sufficiently large $\lambda_q(\Omega)+\eps \geq \lambda_q(\Omega_j)\geq \lambda_q(\Omega)$. Since $\eps>0$ was arbitrary this proves the lemma.
\end{proof}

Next, we turn our attention to the existence of a minimizer for the variational problem~\eqref{def: semi-linear eigenvalue}. For the characterization of cases of equality in the  Brunn--Minkowski inequality for $\lambda_q$ (Theorem \ref{thm: BM lambdaq}), it will be necessary to work under less restrictive assumptions on $\Omega$ than the boundedness and Lipschitz regularity needed for our isoperimetric inequalities.

Assuming that
$$
\lambda_q(\Omega)>0
$$
we deduce that the completion $\mathcal D^{1, 2}_0(\Omega)$ of $C_0^\infty(\Omega)$ with respect to the norm $u\mapsto \|\nabla u\|_{L^2(\Omega)}$ is well defined as a space of almost everywhere defined functions and continuously embedded into $L^q(\Omega)$. Under this assumption the infimum in~\eqref{def: semi-linear eigenvalue} does not change when $C_0^\infty(\Omega)$ is replaced by $\mathcal D^{1, 2}_0(\Omega)$. In what follows, slightly abusing notation, we call a function $u$ in $\mathcal D^{1, 2}_0(\Omega)$ a minimizer if $\|\nabla u\|^2_{L^2(\Omega)} = \lambda_q(\Omega) \|u\|_{L^2(\Omega)}^2>0$.

If $1\leq q<2$, then the assumption $\lambda_q(\Omega)>0$ implies already that the embedding $\mathcal D^{1, 2}_0(\Omega)\subset L^q(\Omega)$ is compact (see~\cite[Theorem 15.6.2]{MazyaSobolevSpaces} and also \cite{BrascoRuffini}) and therefore a minimizer exists. Also, if $|\Omega|<\infty$, then for any $1\leq q\leq 2$ one has $\lambda_q(\Omega)>0$, $\mathcal D^{1, 2}_0(\Omega)=H^1_0(\Omega)$ and the above embedding is compact, so again a minimizer exists.

Whenever there is a minimizer, there is one that is non-negative and normalized in $L^q(\Omega)$. Throughout the paper, $u_{q, \Omega}$ will denote a minimizer with the latter properties. When $1\leq q<2$, there is a unique minimizer with these properties. If $\Omega$ is connected and $q= 2$, minimizers with these properties are again unique; in the multiply connected case $\Omega = \cup_{j\geq 1} \Omega_j$ with connected, open and pairwise disjoint sets $\Omega_j$ we have $\lambda_2(\Omega)=\min_{j\geq 1} \lambda_2(\Omega_j)$ and any minimizer is given by a linear combination of the minimizers for the connected components $\Omega_j$ satisfying $\lambda_2(\Omega)= \lambda_2(\Omega_j)$. In this context we also mention that for $1\leq q<2$ we have
\begin{equation}\label{eq: spin formula}
	u_{q, \Omega} = \sum_{j\geq 1} \Bigl(\frac{\lambda_{q}(\Omega)}{\lambda_q(\Omega_j)}\Bigr)^{\frac{1}{2-q}}u_{q, \Omega_j}\, .
\end{equation}

The following lemma proves a stability property of minimizers.

\begin{lemma}\label{lem: strong eigenfunction conv}
	Fix $1\leq q\leq 2$ and let $\Omega \subset \R^d$ be an open set with finite measure. If $q=2$ assume the minimizer for $\lambda_2(\Omega)$ is unique (up to a multiplicative constant). If $\{u_j\}_{j\geq 1}\subset H_0^1(\Omega)$ satisfies 
	\begin{equation*}
		\lim_{j\to \infty} \frac{\|\nabla u_j\|_{L^2(\Omega)}^2}{\|u_j\|_{L^q(\Omega)}^2} = \lambda_q(\Omega)\, , 
	\end{equation*}
	then $\|u_j\|_{L^q(\Omega)}^{-1} |u_j|\to u_{q, \Omega}$ in $H_0^1(\Omega)$.
\end{lemma}

\begin{proof}
	Set $\tilde u_j = \|u_j\|_{L^q(\Omega)}^{-1} |u_j|$. The sequence $\{\tilde u_j\}_{j\geq 1}$ is bounded in $H_0^1(\Omega)$ and normalized in $L^q(\Omega)$. By passing to a subsequence, we may assume that $\tilde u_j$ converges weakly in $H_0^1(\Omega)$. By the compactness of the embedding $H_0^1(\Omega) \hookrightarrow L^q(\Omega)$, this subsequence converges strongly in $L^q(\Omega)$ to some non-negative limit $\tilde u$. By the sequential weak lower semi-continuity of the Dirichlet energy, we conclude that $\tilde u$ is a non-negative $L^q$-normalized minimizer of the Rayleigh quotient~\eqref{def: semi-linear eigenvalue}. By the assumption on $\Omega$ such minimizers are unique and we conclude that $\tilde u = u_{q, \Omega}$. Moreover, since $\|\nabla \tilde u_{j}\|_{L^2(\Omega)}^2= \lambda_q(\Omega)+o(1)$, we conclude that $\tilde u_{j}$ actually converges to $u_{q, \Omega}$ in $H_0^1(\Omega)$ (see, for instance,~\cite[Proposition~3.32]{Brezis_FunctionalAnalysis}).
\end{proof}

The minimizer $u_{q, \Omega}$ solves the equation 
\begin{equation}\label{eq: nonlinear equation}
	\begin{cases}
		-\Delta u = \lambda_q(\Omega)u^{q-1} & \mbox{in }\Omega\, , \\
		u =0 & \mbox{on }\partial\Omega\, .
	\end{cases}
\end{equation}

We now use elliptic regularity \cite{gilbarg_elliptic_2001} to deduce further information on $u_{q, \Omega}$. 

\begin{lemma}\label{ellipticreg}
	Fix $1\leq q\leq 2$ and let $\Omega\subset\R^d$ be an open set such that $\lambda_q(\Omega)>0$ and such that there is a minimizer $u_{q, \Omega}$. Then $u_{q, \Omega}\in L^\infty(\Omega)\cap C^\infty_{\mathrm{loc}}(\Omega)$. Moreover, if $1\leq q<2$, then $u>0$ in $\Omega$ and, if $q=2$, then in each connected component of $\Omega$ either $u_{q, \Omega}= 0$ or $u_{q, \Omega}>0$.
\end{lemma}

\begin{proof}
	We focus on the case $1<q<2$, for the results are classical in the cases $q=1, 2$ and obtained by similar, but simpler arguments. We write $u=u_{q, \Omega}$. Using, for instance, the technique of Moser iteration, one can infer that $u\in L^\infty(\Omega)$. In fact, there is a constant $C$, depending only on $d, q$, so that
	\begin{equation}\label{eq:linfty}
		\|u_{q, \Omega_j}\|_{L^\infty(\Omega)} \leq C \lambda_q(\Omega_j)^{\frac{d}{2d-q(d-2)}}\, .
	\end{equation}
	This appears as~\cite[Proposition 2.5]{BrascoFranzina_Semilinear_overview_2020}. (The assumption there that the embedding $\mathcal D^{1, 2}_0(\Omega)\subset L^q(\Omega)$ is compact is not necessary for the proof.) Thus, by Riesz potential estimates $u\in C^{1, \alpha}_{\text{loc}}(\Omega)$ for any $\alpha<1$. The local Lipschitz continuity of $u$ and the H\"older-$(q-1)$-continuity of $t\mapsto t^{q-1}$ imply that $u^{q-1}\in C^{0, q-1}_{\text{loc}}(\Omega)$. Thus, by Schauder theory, $u\in C_{\text{loc}}^{2, q-1}(\Omega)$ and the first equation in \eqref{eq: nonlinear equation} holds classically. 
	
	We now show that, if $1\leq q< 2$, then $u>0$ everywhere in $\Omega$, and, if $q=2$, then in each connected component of $\Omega$ either $u>0$ or $u\equiv 0$. If we would know that $u$ extended continuously to $\partial\Omega$, these assertions would be a consequence of the maximum principle, but, since we do not make any assumptions on $\partial\Omega$, we need to argue differently. We first note that $u$ (just like any other function in $C^1_{\text{loc}}(\Omega)$) satisfies, for any $r>0$ and any $x\in\R^d$ with $\dist(x, \Omega^c)<r$, 
	$$
	u(x) = |B_r|^{-1} \int_{B_r(x)} u(y)\, dy + r^2 |B_r|^{-1} \int_{B_r(x)} \nabla_y k((y-x)/r) \cdot \nabla u(y)\, dy \, .
	$$
	Here, for $|z|<1$, 
	$$
	k(z) :=
	\begin{cases}
		(d-2)^{-1} \bigl( d^{-1} |z|^{-d+2} + (2^{-1} -d^{-1}) |z|^2 - 2^{-1} \bigr)
		& \text{if}\ d\neq 2 \, , \\
		2^{-1} \ln(1/|z|) + 4^{-1} |z|^2 - 4^{-1} & \text{if}\ d=2 \, .
	\end{cases}
	$$
	Using the equation for $u$ and the fact that $k(z)=0$ for $|z|=1$ one deduces that
	$$
	u(x) = |B_r|^{-1} \int_{B_r(x)} u(y)\, dy + \lambda_q(\Omega)\, r^2\, |B_r|^{-1} \int_{B_r(x)} k((y-x)/r) u(y)^{q-1}\, dy \, .
	$$
	Assume now that $u(x)=0$ for some $x\in\Omega$. Then this identity, together with the non-negativity and continuity of $u$, implies that $u=0$ in $B_r(x)$. Thus, $\{ x\in\Omega: u(x)=0\}$ is both open and closed in $\Omega$ and, therefore, $u$ vanishes in the connected component of $\Omega$ containing $x$. This is the claimed assertion for $q=2$. For $1\leq q<2$ we obtain a contradiction to \eqref{eq: spin formula}.
	
	Thus, we have shown that $u$ is bounded away from zero on every compact set contained in $\Omega$. Using this information, the regularity of $u$ can be bootstrapped and we finally obtain that $u\in C^\infty_{\text{loc}}(\Omega)$.
\end{proof}

We finally mention that we often extend $u_{q, \Omega}$ by zero to $\R^d\setminus\Omega$ and consider the resulting function on $\R^d$. If $|\Omega|<\infty$, then $u_{q, \Omega}\in H_0^1(\Omega)$ and therefore its extension by zero belongs to $H^1(\R^d)$.

\subsection{The normal derivatives of \texorpdfstring{$u_{q, \Omega}$}{u}}
\label{sec: normal derivative}

Before moving on to our main argument, we show that the normal derivative of $u_{q, \Omega}$ makes sense as an element of $L^2(\partial\Omega)$ under the assumption that $\Omega$ is bounded with Lipschitz boundary. 

If $\Omega$ has $C^{1, \alpha}$-regular boundary for some $0<\alpha<1$, then Schauder theory (see, for instance,~\cite[Theorem 8.33]{gilbarg_elliptic_2001}) implies that $u_{q, \Omega} \in C^{1, \alpha}(\overline \Omega)$ and, in particular, its normal derivative on the boundary is defined in the classical sense and the integral in Theorem~\ref{thm: semi-linear eigenvalue} is well defined. 

To explain the meaning of the normal derivative in Theorem~\ref{thm: semi-linear eigenvalue} for arbitrary bounded open Lipschitz sets, let $\Gamma$ denote the Newtonian kernel in $\R^d$ (see, e.g.~\cite{gilbarg_elliptic_2001}), that is, 
\begin{equation*}
	\Gamma(x) := \begin{cases}
		-\frac{1}{2\pi}\log|x| & d=2\\[3pt]
		c_d|x|^{2-d} & d\neq 2, 
	\end{cases}
\end{equation*}
with a suitable constant $c_d$ and define
\begin{equation*}
  w_{q, \Omega} := -\lambda_q(\Omega)\ \Gamma * u_{q, \Omega}^{q-1} 
\end{equation*}
so that $-\Delta w_{q, \Omega}=\lambda_q(\Omega)u_{q, \Omega}^{q-1}$ in $\R^d$.
Here $u_{q, \Omega}^{q-1}$ should be interpreted as $\1_\Omega$ when $q=1$. (As an aside, we note that, in this case, one could instead of $w_{1, \Omega}$ consider $w = -\lambda_{1}(\Omega)\frac{|x|^2}{2d}\in C^\infty(\R^d)$ and carry out the argument in the same manner.)

Since $u_{q, \Omega}\in L^\infty$ has compact support, we deduce from Riesz potential estimates that $w_{q, \Omega}\in C^{1, \alpha}(\R^d)$ for all $\alpha<1$ (see, for instance,~\cite[Theorem~10.2]{LiebLoss}). In particular, the normal derivative of  $w_{q, \Omega}$ is defined in the classical sense and it suffices to make sense of the normal derivative of the function
\begin{equation*}
   v_{q, \Omega} := u_{q, \Omega}-w_{q, \Omega}
\end{equation*}
Note that this function satisfies
\begin{equation*}
  \begin{cases}
    -\Delta v_{q, \Omega} = 0 \quad & \mbox{in }\Omega\, , \\
    v_{q, \Omega} = -w_{q, \Omega} & \mbox{on }\partial\Omega\, .
  \end{cases}
\end{equation*}
The fact that the normal derivative of $v_{q, \Omega}$ is well-defined and belongs to $L^2(\partial\Omega)$ follows from the results of Jerison--Kenig~\cite{JerisonKenig_BAMS81} and Verchota~\cite{verchota_layer_1984} (see also~\cite[Theorems~5.14.9 \& 5.14.10]{medkova_laplace_2018}). 
Indeed, by these results $\nabla v_{q, \Omega}$ has non-tangential limits almost everywhere on $\partial \Omega$ and the non-tangential maximal function of $|\nabla v_{q, \Omega}|$ belongs to $L^2(\partial\Omega)$. Furthermore, the fact that $u_{q, \Omega}$ vanishes on $\partial\Omega$ implies that the pointwise limit of its gradient is almost everywhere normal to the boundary.

\section{The Brunn--Minkowski inequality for \texorpdfstring{$\lambda_{q}(\Omega)$}{lambda}}
\label{sec: BM inequality}
  
The topic of the current section is the first of the two key ingredients in our strategy for proving Theorem~\ref{thm: semi-linear eigenvalue}, namely, the Brunn--Minkowski-type inequality in Theorem~\ref{thm: BM lambdaq}:
\begin{equation*}
   \lambda_{q}((1-s)\Omega_0+s\Omega_1) \leq \Bigl((1-s)\lambda_{q}(\Omega_0)^{-\alpha_q}+s\lambda_{q}(\Omega_{1})^{-\alpha_q}\Bigr)^{-1/\alpha_q} \, .
 \end{equation*}
This inequality is essentially due to Brascamp--Lieb~\cite{BrascampLieb_JFA76} for $q=2$, Borell~\cite{Borell_85} for $q=1$, and Colesanti~\cite{colesanti_brunnminkowski_2005} for $1\leq q<2$. However, the statements in these references are under slightly stronger assumptions than what we state here. Indeed, Colesanti~\cite{colesanti_brunnminkowski_2005} assumes the sets $\Omega_0, \Omega_1$ to be bounded with $C^2$-regular boundary, Brascamp and Lieb~\cite{BrascampLieb_JFA76} impose that the sets be connected and have finite measure, and Borell~\cite{Borell_85} assumes the sets to be bounded. That being said, deducing the inequality for arbitrary non-empty open sets from these results is not difficult. 

In \cite{colesanti_brunnminkowski_2005}, Colesanti characterized the cases of equality in the Brunn--Minkowski-type inequality for $\lambda_q$ for $1\leq q<2$ under the assumption that the sets are $C^2$ and bounded. As stated in Theorem \ref{thm: BM lambdaq}, in this paper we will show that these additional assumptions on the sets are not necessary.
Since this is somewhat technical and not central to the core topic of this paper, we defer its proof to Appendix~\ref{Appendix}.

Before we prove that the inequality for arbitrary non-empty open sets follows from the known results, we note that Colesanti~\cite{colesanti_brunnminkowski_2005} and Borell~\cite{Borell_85} do not state their results in terms of $\lambda_q(\Omega)$ but in terms of the quantity
\begin{equation*}
    {\bf F}(\Omega) := \int_\Omega |\nabla \tilde u(x)|^2\, dx
  \end{equation*}
  where $\tilde u$ is the unique positive solution of the (perhaps simpler looking) PDE
  \begin{equation*}
     \begin{cases}
       -\Delta \tilde u = \tilde u^{q-1} & \mbox{in }\Omega\, , \\
       \tilde u = 0 & \mbox{on }\partial \Omega\, .
     \end{cases}
  \end{equation*} 
  By multiplying the solution with a constant, one deduces that $\tilde u = \lambda_q(\Omega)^{\frac{1}{q-2}}u_{q, \Omega}$, and ${\bf F}(\Omega) = \lambda_q(\Omega)^{\frac{q}{q-2}}$ and the inequality of Theorem~\ref{thm: BM lambdaq} is equivalent to a corresponding inequality for ${\bf F}(\Omega)$. Colesanti's (and Borell's) choice of working with ${\bf F}$ is more natural when $q=1$ since the PDE is then the classical torsion problem. Our formulation is adapted to also naturally encompass $q=2$, where the PDE becomes linear and the dependence on~$\lambda_2$ cannot be eliminated by multiplying by a constant.

Let us first treat the trivial case when $\min\{\lambda_q(\Omega_0), \lambda_q(\Omega_1)\}=0$. For any $t\in (0, 1)$ the set $(1-t)\Omega_0 + t\Omega_1$ contains rescaled and translated copies of both $\Omega_0$ and $\Omega_1$. Therefore, by domain monotonicity, translation invariance and scaling homogeneity of $\lambda_q$ it follows that $\lambda_q((1-t)\Omega_0+t\Omega_1)=0$.

\subsection{The case \texorpdfstring{$q=2$}{q=2}}

In the case $q=2$, the inequality in Theorem~\ref{thm: BM lambdaq} can be deduced from an inequality due to Brascamp and Lieb~\cite{BrascampLieb_JFA76}. Indeed, in that paper they prove that for open and connected sets $\Omega_0, \Omega_1\subset \R^d$ with finite measure, 
\begin{equation}\label{eq: logconcavity heat trace}
  \Tr(e^{t\Delta_{(1-s)\Omega_0+s\Omega_1}}) \geq \bigl(\Tr(e^{t\Delta_{\Omega_0}})\bigr)^{1-s}\bigl(\Tr(e^{t\Delta_{\Omega_1}})\bigr)^{s}\, , \quad \mbox{for all } s\in [0, 1] \mbox{ and } t>0\, .
\end{equation}
Taking the logarithms, dividing by $t$, and letting $t\to \infty$, Brascamp and Lieb showed that the inequality~\eqref{eq: logconcavity heat trace} and the positivity of the spectral gap implies
\begin{equation}\label{eq: weaker BM lambda}
   \lambda_2((1-s)\Omega_0+s\Omega_1) \leq (1-s)\lambda_2(\Omega_0)+s\lambda_2(\Omega_1)\, .
 \end{equation}
Our next goal is to prove two things; first, that the assumption that $\Omega_0, \Omega_1$ are connected can dropped, and secondly, that~\eqref{eq: weaker BM lambda} implies the inequality in Theorem~\ref{thm: BM lambdaq} with $q=2$.  

To deduce the statement for sets with multiple connected components, one can argue as follows. Let $\Omega_0 = \cup_{j\in I_0} \Omega_0^j, \Omega_1 = \cup_{j\in I_1} \Omega_1^j$, with $\Omega_i^j \cap \Omega_i^{j'}=\emptyset$ for $i\in \{0, 1\}$ and $(j, j')\in (I_i\times I_i) \setminus \{j, j\}$. It holds that $\lambda_2(\Omega_i)= \min_{j\in I_i}\lambda_2(\Omega_i^j)$. Since
 \begin{equation}\label{eq: Minksum disjoint unions}
    (1-s)\Omega_0 + s\Omega_1 = \bigcup_{(j, j')\in I_0 \times I_1} \Bigl[(1-s)\Omega_0^j + s\Omega_1^{j'}\Bigr] \, , 
  \end{equation}
  the monotonicity under inclusion implies that
  \begin{equation*}
    \lambda_2((1-s)\Omega_0+s\Omega_1)\leq \min_{(j, j')\in I_0\times I_1} \lambda_2((1-s)\Omega_0^j+s\Omega_1^{j'})\, .
  \end{equation*} 
  The desired inequality follows by applying the inequality for each of the pairs $(\Omega_0^j, \Omega_1^{j'})$. We emphasize that while the sets $\{\Omega_i^j\}_{j\in I_i}$ are assumed to be disjoint, this is not generally the case for the sets in the union~\eqref{eq: Minksum disjoint unions}.

The fact that the a priori weaker inequality~\eqref{eq: weaker BM lambda} implies the inequality in Theorem~\ref{thm: BM lambdaq} with $q=2$ can be proved by the following argument (which is somewhat standard in the field but we include it for completeness). Set, for $j=0, 1$, 
\begin{equation*}
  \omega_j := \lambda_2(\Omega_j)^{1/2}\Omega_j
\end{equation*}
so that $\lambda_2(\omega_j)=1$. Then
\begin{equation*}
  (1-s)\Omega_0 + s\Omega_1 = ((1-s)\lambda_2(\Omega_0)^{-1/2}+s\lambda_2(\Omega_1)^{-1/2})((1-\tilde s)\omega_0 + \tilde s \omega_1)
\end{equation*}
with
\begin{equation*}
  \tilde s := \frac{s\lambda_2(\Omega_1)^{-1/2}}{(1-s)\lambda_2(\Omega_0)^{-1/2}+ s \lambda_2(\Omega_1)^{-1/2}} \in [0, 1]\, .
\end{equation*}
Thus, by the homogeneity of $\lambda_2$ and applying~\eqref{eq: weaker BM lambda} to $(1-\tilde s)\omega_0+ \tilde s \omega_1$, we deduce
\begin{align*}
  \lambda_2((1-s)\Omega_0+s\Omega_1)
  &=
  ((1-s)\lambda_2(\Omega_0)^{-1/2}+s\lambda_2(\Omega_1)^{-1/2})^{-2}\lambda_2((1-\tilde s)\omega_0 + \tilde s \omega_1)\\
  &\leq
  ((1-s)\lambda_2(\Omega_0)^{-1/2}+s\lambda_2(\Omega_1)^{-1/2})^{-2}\, .
\end{align*}
This proves Theorem~\ref{thm: BM lambdaq} for $q=2$.


\subsection{The general case \texorpdfstring{$1\leq q\leq 2$}{}}

We now prove that the assumptions of boundedness and regularity of the boundary of $\Omega_0, \Omega_1$ can be dropped in the works of  Borell~\cite{Borell_85} and Colesanti~\cite{colesanti_brunnminkowski_2005}. The same argument removes the remaining assumption that the measures of $\Omega_0, \Omega_1$ are finite for $q=2$.

Let $\Omega\subset \R^d$ be open and non-empty. There are bounded open sets $\Omega_j$ with $C^\infty$-regular boundary such that $\Omega_j \subset \Omega_{j+1}$, $\cup_j \Omega_j = \Omega$, and $\Omega_j \cap B_R$ converges with respect to the Hausdorff distance to $\Omega\cap B_R$ for any $R>0$. By Lemma~\ref{lem: Hausdorff interior continuity}, 
\begin{equation*}
  \lim_{j \to \infty}\lambda_q(\Omega_j) = \lambda_q(\Omega)\, .
\end{equation*}
In particular, for $\Omega_0, \Omega_1$ as in Theorem~\ref{thm: BM lambdaq} there are smooth exhaustions $\{\Omega_0^j\}_{j\geq 1}, \{\Omega_1^j\}_{j\geq 1}$ satisfying the assumptions of Lemma~\ref{lem: Hausdorff interior continuity}. For each $j\geq 1$ and $s\in [0, 1]$, we have $(1-s)\Omega_0+s\Omega_1\supset (1-s)\Omega_0^j + s\Omega_1^j$ and therefore, by the domain monotonicity of $\lambda_q$ and the validity of the Brunn--Minkowski inequality for smooth sets, 
\begin{equation*}
  \lambda_q((1-s)\Omega_0+s\Omega_1) \leq \lambda_q((1-s)\Omega_0^j + s\Omega_1^j) \leq \Bigl((1-s)\lambda_q(\Omega_0^j)^{-\alpha_q}+s \lambda_q(\Omega_1^j)^{-\alpha_q}\Bigr)^{-1/\alpha_q}\, .
\end{equation*}
By Lemma~\ref{lem: Hausdorff interior continuity}, the claimed inequality follows by sending $j \to \infty$.


\section{A Hadamard variational formula for \texorpdfstring{$\lambda_q(\Omega)$}{lambda}}
\label{sec: Hadamard}

In this and the next section, we provide the second key ingredient in the proof of our main result. Here, we shall prove the following Hadamard variational formula for $\lambda_q$.

\begin{theorem}\label{thm: variation}
	Let $1\leq q\leq 2$, let $\Omega \subset \R^d$ be an open set of finite measure and assume that there is a unique non-negative minimizer $u_{q, \Omega}$ for $\lambda_q(\Omega)$ which is normalized in $L^q(\Omega)$.
  Let $\Phi\in C^1((-T, T); W^{1, \infty}(\R^d; \R^d))$, $T>0$, be such that, for all $x\in\R^d$, $\Phi(0, x) = x$ and, for all $t\in (-T, T)$, the map $\Phi(t, \, \cdot\, )\colon \R^d \to \R^d$ is a bi-Lipschitz homeomorphism of an open neighborhood of $\Omega$ onto its image. Set $\dot\Phi = \partial_t\Phi|_{t=0}$ and let $D\dot \Phi$ be the Jacobian of~$\dot\Phi$. Then
  \begin{align}\label{eq: hadamards formula0}
    \lim_{t\to 0} \frac{\lambda_q(\Phi(t, \Omega))-\lambda_q(\Omega)}{t} & = -2 \int_\Omega \nabla u_{q, \Omega} \cdot \bigl( D\dot\Phi \bigr) \nabla u_{q, \Omega}\, dx \notag \\
    & \quad + \int_\Omega \biggl( |\nabla u_{q, \Omega}|^2 - \frac{2}{q}\lambda_q(\Omega) u_{q, \Omega}^q \biggr) \nabla\cdot\dot\Phi\, dx \, , 
  \end{align}
  If, in addition, $\Omega$ has Lipschitz boundary, then
	\begin{align}\label{eq: hadamards formula}
		\lim_{t\to 0} \frac{\lambda_q(\Phi(t, \Omega))-\lambda_q(\Omega)}{t} & = - \int_{\partial\Omega} \Bigl( \frac{\partial u_{q, \Omega}}{\partial\nu}\Bigr)^2 \nu\cdot\dot\Phi\, d \mathcal H^{d-1}(x) \, , 
	\end{align}
	where $\nu$ denotes the outward pointing unit normal field on $\partial\Omega$.
\end{theorem}

We recall from Section~\ref{sec: Prelminaries} that the assumption that there is a unique (up to a multiplicative constant) minimizer for $\lambda_q(\Omega)$ is automatically satisfied for $1\leq q<2$. If $q=2$, it is satisfied if and only if there is a unique connected component $\Omega_j$ of $\Omega = \cup_{j\geq 1} \Omega_j$ for which $\lambda_2(\Omega_j)$ is minimal. In particular, it is satisfied if $\Omega$ is connected.

For $q=2$ the formula for the first variation of the eigenvalue is well-known and due to Hadamard~\cite{Ha}; see, for instance,~\cite[Theorem 5.7.1]{MR3791463} for a textbook presentation. Similarly, for $q=1$ a change of variables relates the first variation of $\lambda_q(\Omega)$ to that of the torsional rigidity, which can be found, for instance, in~\cite[Equation (5.103)]{MR3791463}. For $1<q<2$ we have not been able to find the result in the existing literature. 

In the cases $q=1, 2$, our proof is different from the standard proof presented, e.g.\ in~\cite{MR3791463}. From a conceptual point of view, these standard proofs establish, at the same time as establishing the differentiability of $\lambda_q(\Phi(t, \Omega))$, the differentiability of $u_{q, \Phi(t, \Omega)}$. The equation for the derivative of $u_{q, \Phi(t, \Omega)}$ is then used to derive a formula for $\lambda_q(\Phi(t, \Omega))$. Our approach completely bypasses the differentiability of $u_{q, \Phi(t, \Omega)}$. From a technical point of view, the standard proof of Hadamard formulas for $q=1, 2$ relies on the implicit function theorem, but it is not clear to us how to apply this because of the non-differentiability of $u\mapsto u^{q-1}$ at $u=0$ for $1<q<2$. Instead, our argument has a variational character.

\begin{lemma}\label{uniquevar}
	Let $\Omega$ and $\Phi$ be as in Theorem~\ref{thm: variation}. Then, for all sufficiently small $|t|$, there is a unique (up to multiplication by a constant) minimizer for $\lambda_2(\Phi(t, \Omega))$. 
\end{lemma}

\begin{proof}
	Let $\Omega=\cup_{j\geq 1}\Omega_j$ with disjoint, connected open sets $\Omega_j$.
	For all $t\in (-T, T)$, since $\Phi(t, \cdot)$ is a homeomorphism, the sets $\Phi(t, \Omega_j)$ are disjoint, connected open sets. Since $\Phi(t, \Omega_j)$ are connected there is a unique normalized, non-negative minimizer $u_{2, \Phi(t, \Omega_j)}$ for $\lambda_2(\Phi(t, \Omega_j))$. Taking $u_{2, \Phi(t, \Omega_j)}\circ \Phi(t, \cdot)$ and $u_{2, \Omega_j}\circ \Phi(t, \cdot)^{-1}$ as trial functions in the variational characterizations of $\lambda_2(\Omega_j)$ and $\lambda_2(\Phi(t, \Omega_j))$ one can prove that $\lambda_2(\Phi(t, \Omega_j))\to\lambda_2(\Omega_j)$ as $t\to 0$ for each $j$ (for details see the proof of Theorem~\ref{thm: variation}).

	If $\Omega$ has only finitely many connected components, then this implies the assertion. Indeed, the uniqueness of a minimizer for $\lambda_2(\Phi(t, \Omega))$ is equivalent to there being a single $j$ for which the infimum over $\lambda_2(\Phi(t, \Omega_j))$ is achieved. If there is a unique minimizing $j_0$ at $t=0$, then, by the above convergence of $\lambda_2(\Phi(t, \Omega_j))$, the same $j_0$ is also minimizing for all $|t|$ sufficiently small.
	
	If $\Omega$ has infinitely many connected components, we need an additional argument to control its small components. The assumption on $\Phi$ implies that its Jacobian converges uniformly to $1$. Thus, $|\Phi(t, \Omega_j)|/|\Omega_j|\to 1$ uniformly in $j$. From this and the (not necessarily sharp) Faber--Krahn inequality (see, e.g.~\cite{MR2251558}) we conclude that there is a $T_0\in(0, T)$ such that for all $j$ and all $|t|\leq T_0$, 
	$$
	\lambda_2(\Phi(t, \Omega_j)) \geq C_d |\Phi(t, \Omega_j)|^{-2/d} \geq (C_d/2) |\Omega_j|^{-2/d} \, .
	$$	
	Thus, for the question whether $\lambda_2(\Phi(t, \Omega_j))$ is attained at a unique $j$ it suffices to consider~$j$ with $(C_d/2) |\Omega_j|^{-2/d} < (1/2) \lambda_2(\Omega)$. Since $\Omega$ has finite measure, this is a finite number, and the proof can be concluded as before.
\end{proof}

\begin{proof}[Proof of Theorem~\ref{thm: variation}]
	We abbreviate
	$$
	\Omega(t):=\Phi(t, \Omega)
	\qquad\text{and}\qquad
	\lambda(t):=\lambda_q(\Omega(t)) \, .
	$$
	According to our discussion in Section~\ref{sec: Prelminaries}, for $1\leq q<2$, the normalized, non-negative minimizers $u_{q, \Omega(t)}$ of $\lambda(t)$ are unique. The same is true for $q=2$, provided $|t|$ is small, by the assumption of the theorem and Lemma~\ref{uniquevar}. We abbreviate
	$$
	u_{t}:= u_{q, \Omega(t)} \, .
	$$
	Define also $v_t\colon \Omega \to \R$ by
	$$
	v_t := u_t \circ \Phi(t, \cdot) \, .
	$$
	Since $u_{t}$ is non-negative and normalized in $L^q(\Omega(t))$, $v_t$ is non-negative and satisfies
	\begin{equation*}
		\int_\Omega v_{t}^{q} J_t\, dx = 1
	\end{equation*}
	with
	$$
	J_t := |{\det(D_x\Phi(t, \cdot))}| \, .
	$$
	Here and in what follows we write $D_x \Phi(t, \cdot)$ for the Jacobian of the map $x \mapsto \Phi(t, x)$.
	Since $D_{x}\Phi(t, \cdot)$ is bounded, we have $v_t\in H^1_0(\Omega)$ and
	$$
	\lambda(t) = \int_{\Omega(t)} |\nabla u_t|^2\, dx = \int_\Omega \nabla v_t \cdot A_t \nabla v_t\, dx
	$$
	with
	$$
	A_t := J_t  (D_x\Phi(t, \cdot))^{-1} ((D_x\Phi(t, \cdot))^{-1})^\top \, .
	$$

	After these preparations, we now start with the main argument. Since $u_0\circ\Phi(t, \, \cdot\, )^{-1}\in H_0^1(\Omega(t))$ and $v_t\in H_0^1(\Omega)$, these functions can be taken as trial functions in the variational characterizations of $\lambda(t)$ and $\lambda(0)$, respectively, which implies that
	\begin{equation}
		\label{eq:trialfcns}
		\lambda(t) \leq \frac{\int_\Omega \nabla u_0\cdot A_t \nabla u_0\, dx}{\bigl( \int_\Omega u_0^q J_t \, dx \bigr)^{2/q}}
		\quad\text{and}\quad
		\lambda(0) \leq \frac{\int_\Omega |\nabla v_t|^2 \, dx}{\bigl( \int_\Omega v_t^q \, dx \bigr)^{2/q}} \, .
	\end{equation}
	It follows from $A_t\to \id$ and $J_t\to 1$ in $L^\infty$ that
	$$
	\frac{\int_\Omega \nabla u_0\cdot A_t \nabla u_0\, dx}{\bigl( \int_\Omega u_0^q J_t \, dx \bigr)^{2/q}} = (1+o(1)) \frac{\int_\Omega |\nabla u_0|^2\, dx}{\bigl( \int_\Omega u_0^q \, dx \bigr)^{2/q}} = (1+o(1)) \lambda(0)
	$$
	and
	$$
	\frac{\int_\Omega |\nabla v_t|^2 \, dx}{\bigl( \int_\Omega v_t^q \, dx \bigr)^{2/q}}
	= (1+o(1)) \frac{\int_\Omega \nabla v_t\cdot A_t \nabla v_t\, dx}{\bigl( \int_\Omega v_t^q J_t \, dx \bigr)^{2/q}} = (1+o(1)) \lambda(t) \, .
	$$
	Thus, we have shown that $\lambda(t)\leq (1+o(1))\lambda(0)$ and $\lambda(0)\leq (1+o(1))\lambda(t)$ and, therefore, $\lambda(t)\to \lambda(0)$. Moreover, we conclude that
\begin{equation*}
	\lim_{t\to 0} \|\nabla v_t\|_{L^2(\Omega)}^2 = \lambda(0) \quad \mbox{and} \quad \lim_{t\to 0} \|v_t\|_{L^q(\Omega)}^2 =1\, , 
\end{equation*}
and, therefore, by Lemma~\ref{lem: strong eigenfunction conv}, $v_t \to u_0$ in $H^1_0(\Omega)$.

With this information at hand, we return to~\eqref{eq:trialfcns}, which we rewrite as
$$
\lambda(t) \leq \frac{\lambda(0) + t n(t)}{( 1+ t d(t))^{2/q}}
\quad\text{and}\quad
\lambda(0) \leq \frac{\lambda(t) - t \tilde n(t)}{( 1- t \tilde d(t))^{2/q}} \, , 
$$
where we used the normalizations of $u_0$ and $v_t$ and set
$$
n(t) := \int_\Omega \nabla u_0\cdot \bigl( t^{-1}( A_t-\id)\bigr) \nabla u_0\, dx \, , 
\qquad
\tilde n(t) := \int_\Omega \nabla v_t \cdot \bigl( t^{-1}( A_t-\id)\bigr) \nabla v_t\, dx \, , 
$$
and
$$
d(t) := \int_\Omega u_0^q\ t^{-1}(J_t-1)\, dx \, , 
\qquad
\tilde d(t) := \int_\Omega v_t^q\ t^{-1}(J_t-1)\, dx \, .
$$
The assumption $D\Phi(t, \cdot) = \id + t D\dot\Phi + o(t)$ in $L^\infty(\R^d, \R^{d\times d})$ implies that
\begin{align*}
	t^{-1} (A_t-\id) & \to -D\dot\Phi - (D\dot\Phi)^\top + \nabla\cdot \dot\Phi =: \dot A_0 
	\qquad\text{in}\ L^\infty(\R^d, \R^{d\times d}) \, , \\
	t^{-1} (J_t - 1) & \to \nabla\cdot \dot\Phi =: \dot J_0
	\qquad\text{in}\ L^\infty(\R^d, \R) \, .
\end{align*}
(Of course, in the limit defining $\dot A_0$, $\nabla\cdot\dot\Phi$ is identified with $\nabla\cdot\dot\Phi$ times the identity matrix.) This, together with the fact that $v_t\to u_0$ in $H^1_0(\Omega)$ (and therefore also in $L^q(\Omega)$), implies that
$$
n(t) \to n_0 \quad\text{and}\quad \tilde n(t) \to n_0 \, , 
\qquad\text{where}\ n_0 := \int_\Omega \nabla u_0\cdot \dot A_0 \nabla u_0\, dx \, , 
$$
and
$$
d(t) \to d_0 \quad \text{and}\quad \tilde d(t)\to d_0 \, , 
\qquad\text{where}\ d_0 :=\int_\Omega u_0^q\ \dot J_0 \, dx \, .
$$
Thus, we have shown that
\begin{align*}
	\lambda(t) & \leq \frac{\lambda(0) + t n_0 + o(t)}{( 1+ td_0 +o(t))^{2/q}} = \lambda(0) + t \biggl( n_0 - \frac 2q \lambda(0) d_0 \biggr) + o(t)  \, , \\
	\lambda(0) & \leq \frac{\lambda(t) - t n_0 + o(t)}{( 1- t d_0 + o(t))^{2/q}} = \lambda(t) - t \biggl( n_0 - \frac{2}{q} \lambda(t) d_0 \biggr) + o(t) \\
	& = \lambda(t) - t \biggl( n_0 - \frac{2}{q} \lambda(0) d_0 \biggr) + o(t) \, .
\end{align*}
We conclude that
$$
\frac{\lambda(t) - \lambda(0)}{t} = n_0 - \frac 2q \lambda(0) d_0 + o(1) \, , 
$$
that is, $\lambda$ is differentiable at $0$ with derivative given by~\eqref{eq: hadamards formula0}.

Assume now that $\Omega$ has Lipschitz boundary. In order to bring the derivative into the form~\eqref{eq: hadamards formula} we note that, by elliptic regularity (Lemma \ref{ellipticreg}), $u_0$ is smooth in $\Omega$. This, together with the existence of boundary values of $\nabla u_0$ discussed in Section~\ref{sec: Prelminaries}, implies
\begin{align}\label{eq:niceform}
	d_0 & = - q \int_\Omega u_0^{q-1} \dot\Phi \cdot \nabla u_0 \, dx = \frac{q}{\lambda(0)} \int_\Omega (\Delta u_0) \dot\Phi \cdot \nabla u_0 \, dx \notag \\
	& = - \frac{q}{\lambda(0)} \int_\Omega \nabla u_0 \cdot\nabla \bigl(\dot\Phi \cdot \nabla u_0 \bigr) dx + \frac{q}{\lambda(0)} \int_{\partial\Omega} \nu\cdot\nabla u_0 \dot \Phi\cdot \nabla u_0 \, d\mathcal H^{d-1}(x) \, .
\end{align}
For the first term on the right side we use
$$
\nabla u_0 \cdot\nabla \bigl(\dot\Phi\cdot\nabla u_0 \bigr) = \frac12 \Bigl( \nabla u_0 \cdot \bigl( D\dot\Phi + (D\dot\Phi)^\top\bigr) \nabla u_0 + \dot\Phi \cdot \nabla (|\nabla u_0|^2) \Bigr)
$$
and obtain, integrating by parts, 
$$
\int_\Omega \nabla u_0 \cdot\nabla \bigl(\dot\Phi \cdot \nabla u_0 \bigr) dx 
= - \frac12 \biggl( n_0 - \int_{\partial\Omega} |\nabla u_0|^2 \nu\cdot\dot\Phi\, d\mathcal H^{d-1}(x) \biggr)\,.
$$
Inserting this into~\eqref{eq:niceform} and using $\nabla u_0 = (\partial u_0/\partial\nu) \nu$ on $\partial\Omega$, we obtain
\begin{align*}
	d_0 & = \frac{q}{2\lambda(0)} \biggl( n_0 - \int_{\partial\Omega}\Bigl( \frac{\partial u_0}{\partial\nu} \Bigr)^2 \nu\cdot\dot\Phi \, d\mathcal H^{d-1}(x) \biggr) + \frac{q}{\lambda(0)} \int_{\partial\Omega} \Bigl( \frac{\partial u_0}{\partial\nu} \Bigr)^2 \nu\cdot\dot\Phi\, d\mathcal H^{d-1}(x) \\
	& = \frac{q}{2\lambda(0)} \, n_0 + \frac{q}{2\lambda(0)} \int_{\partial\Omega} \Bigl( \frac{\partial u_0}{\partial\nu} \Bigr)^2 \nu\cdot\dot\Phi\, d\mathcal H^{d-1}(x) \, .
\end{align*}
This implies the form~\eqref{eq: hadamards formula} of the derivative and concludes the proof of Theorem~\ref{thm: variation}.
\end{proof}


\section{Approximation of the Minkowski sum for \texorpdfstring{$C^1$}{C1} sets}
\label{sec: Geometry}

The aim of this section is to prove Theorem~\ref{thm: Hadamard Minkowski sum}, that is, for $\Omega\subset \R^d$ open, bounded and connected with $C^1$ boundary we wish to show that
\begin{equation*}
   \lim_{t \to 0^\limplus} \frac{\lambda_q(\Omega + tB)-\lambda_q(\Omega)}{t} = -\int_{\partial \Omega} \Bigl(\frac{\partial u_{q, \Omega}}{\partial \nu}\Bigr)^2\, d\Haus^{d-1}(x)\, .
 \end{equation*}
To achieve this we shall argue that the Minkowski sum $\Omega + tB$ can be approximated both from the interior and exterior by the image of $\Omega$ under a diffeomorphism. This, in turn, will allow us to apply the Hadamard variational formula in Theorem~\ref{thm: variation}, which a priori does not cover the variation induced by taking the Minkowski sum.

Define the signed distance function $\delta_\Omega$ by
\begin{equation*}
  \delta_\Omega(x) := \dist(x, \Omega)-\dist(x, \Omega^c)\, .
\end{equation*}
Here, we use the convention that $\delta_\Omega$ is negative in $\Omega$ and positive in $\Omega^c$. Recall that for any $\Omega \subset \R^d$ it holds that $|\nabla \delta_\Omega(x)|\leq 1$ for almost every $x\in \R^d$. Moreover, $\Omega+ tB$ is the sub-level set $\{x\in \R^d: \delta_\Omega(x)<t\}$. If $\nabla\delta_\Omega$ is Lipschitz in a neighborhood of $\partial\Omega$, the flow map associated with this vector field is a bijection from this neighborhood onto its image. As such, if $\delta_\Omega$ is sufficiently regular in a neighborhood of $\partial\Omega$ to allow for application of Hadamard's variational formula with the associated flow map, the statement of Theorem~\ref{thm: semi-linear eigenvalue} would follow in a straightforward manner. To handle boundaries of low regularity we shall follow the same idea but in combination with a mollification argument.

As seen in the previous section, a Hadamard variational formula is not so much dependent on the regularity of $\Omega$ as the regularity of the map $\Phi \in C^1((-T, T); W^{1, \infty}(\R^d; \R^d))$. Indeed the Lipschitz assumption in Theorem~\ref{thm: variation} is only used to justify the use of Green's identities and to make sense of the normal derivative of the minimizer as an element of $L^2(\partial\Omega)$. However, when it comes to the variational formula in Theorem~\ref{thm: Hadamard Minkowski sum} the regularity of the perturbation is intimately connected with the regularity of the underlying set $\Omega$.

 \subsection{Construction of approximate mapping}

 For an open set $\Omega \subset \R^d$, $\eps_0>0, \eta_0 >0$, we define $\Phi_\Omega\colon (-1, 1)\times \R^d \to \R^d$ by
 \begin{align*}
  \Phi_\Omega(t, x) 
  :=
  \Phi^{\eps_0, \eta_0}_\Omega(t, x) 
  &: = 
  x+ t \eps_0^{-d}\int_{\R^d} \varphi\Bigl(\frac{|y|}{\eps_0}\Bigr)\psi\Bigl(\frac{\delta_\Omega(x-y)}{\eta_0}\Bigr)\nabla\delta_\Omega(x-y)\, dy\\
  &= 
  x+ t \eps_0^{-d}\int_{\R^d} \varphi\Bigl(\frac{|x-y|}{\eps_0}\Bigr)\psi\Bigl(\frac{\delta_\Omega(y)}{\eta_0}\Bigr)\nabla\delta_\Omega(y)\, dy\\
  &=: x+ tX(x)\, , 
 \end{align*}
 where $\varphi \in C^\infty([0, 1])$ is non-increasing with $\varphi(1)=0$, $\varphi'\in C_0^\infty((0, 1))$, and satisfies $|\S^{d-1}|\int_0^1 \varphi(y)y^{d-1}\, dy =1$, while $\psi\in C_0^\infty((-1, 1))$ satisfies $0\leq \psi\leq 1$ and $\psi \equiv 1$ in $[-1/2, 1/2]$.

The key observation of this section is the following geometric result, which might be of independent interest.
\begin{proposition}\label{prop: set inclusion}
  Let $\Omega \subset \R^d$ be an open set and define $\Phi_\Omega = \Phi_\Omega^{\eps_0, \eta_0}, X$ as above. Then
  \begin{enumerate}[label=(\alph*)]
    \item for any $\eps_0>0$, if $|t|$ is sufficiently small, then for any $\eta_0>0$ the map $\Phi_\Omega(t, \cdot)$ is a diffeomorphism of $\R^d$ onto itself.

    \item for any $\eps_0, \eta_0 >0$ and $t\in [0, 1)$, 
  \begin{equation*}
    \Phi_\Omega(t, \Omega) \subseteq \Omega + tB\, .
  \end{equation*}

  \item\label{it: approx inclusion} if $\Omega$ is bounded, has $C^{1}$-regular boundary, and $\delta, \eta_0>0$, then there is an $\eps_0>0$ small enough so that for all sufficiently small $t>0$, 
  \begin{equation*}
    \Phi_\Omega((1+\delta)t, \Omega)\supset \Omega + tB\, .
  \end{equation*}  

  \item\label{it: normal convergence} if $\Omega$ is bounded, has $C^{1}$-regular boundary, and $\delta, \eta_0>0$, then
  \begin{equation*}
    \|\dot\Phi_\Omega -\nu\|_{L^\infty(\partial\Omega)} = \|X-\nu\|_{L^\infty(\partial\Omega)}= o_{\eps_0 \to 0}(1)\, , 
  \end{equation*}
  where $\dot\Phi_\Omega = \partial_t\Phi_\Omega|_{t=0}$ and $\nu$ denotes the outward pointing unit normal field on $\partial\Omega$.
  \end{enumerate} 
\end{proposition}

\begin{remark}
  The third and fourth part of the proposition do not extend to general Lipschitz sets. In fact,~\ref{it: approx inclusion} fails for planar polygons. Indeed, it is an easy computation to see that if $\Omega \subset \R^2$ is a polygon and $0 \in \partial \Omega$ is a corner of interior angle $\theta$, then $|\Phi_\Omega(t, 0)| = t \sqrt{2-2\cos(\theta)}/2$ for $t>0$ as long as the $\eps_0$ ball around $0$ contains no other corner of~$\Omega$. Thus, if $\theta \neq \pi$ (a ``flat corner'') this point is mapped to the interior of $\Omega + t'B$ unless $t > c(\theta) t'$ with $c(\theta)>1$. In particular, this proves that we cannot take $\delta$ arbitrarily close to $0$. Similarly, it is proved by Hofmann, Mitrea, and Taylor~\cite{hofmann_geometric_2007} that, if we define, for $\Omega \subset \R^d$ open and bounded, 
  \begin{equation*}
    \rho(\Omega) :=\inf\{\|X-\nu\|_{L^\infty(\partial\Omega)}: X\in C^0(\partial\Omega; \R^d), |X|=1 \mbox{ on }\partial\Omega\}\, , 
  \end{equation*}
  then
  \begin{align*}
    \rho(\Omega) = 0 \quad &\Longleftrightarrow \quad \partial \Omega \mbox{ is } C^1\, , \\
    \rho(\Omega) < \sqrt{2} \quad &\Longleftrightarrow \quad \partial \Omega \mbox{ is Lipschitz}\, .\\
  \end{align*}
  In particular, this implies that the validity of~\ref{it: normal convergence} for $X$ that is merely continuous on $\partial\Omega$ implies that $\partial\Omega$ is $C^1$. We note that the notion of Lipschitz sets used here (and frequently in the mathematics literature) is in~\cite{hofmann_geometric_2007} referred to as \emph{strongly} Lipschitz to distinguish it from the somewhat less commonly occurring notion of weakly Lipschitz sets.
\end{remark}

\begin{proof}[Proof of Proposition~\ref{prop: set inclusion}]
  For notational convenience, throughout the proof we drop the subscript $\Omega$ for the mapping $\Phi$.

  To prove the first claim, it suffices to prove that $\Phi$ is injective. We argue by contradiction. Fix $t$ and let $x_1, x_2$ be such that $x_1 \neq x_2$ and $\Phi(t, x_1)=\Phi(t, x_2)$. Then, by the definition of~$\Phi$ and the fundamental theorem of calculus, 
  \begin{align*}
    |x_1-x_2| 
    &= 
    |t| \eps_0^{-d}\biggl|\int_{\R^d}\Bigl(\varphi\Bigl(\frac{|x_1-y|}{\eps_0}\Bigr)-\varphi\Bigl(\frac{|x_2-y|}{\eps_0}\Bigr)\Bigr)\psi\Bigl(\frac{\delta_\Omega(y)}{\eta_0}\Bigr)\nabla\delta_\Omega(y)\, dy\biggr|\\
    &= 
    |t| \eps_0^{-d}\biggl|\int_{\R^d}\biggl(\int_0^1\varphi'\Bigl(\frac{|\rho x_1+(1-\rho)x_2-y|}{\eps_0}\Bigr)\frac{x_1-x_2}{\eps_0}\, d\rho\biggr)\psi\Bigl(\frac{\delta_\Omega(y)}{\eta_0}\Bigr)\nabla\delta_\Omega(y)\, dy\biggr|\\
    &\leq 
    |t| \eps_0^{-1}C_d\|\varphi'\|_{\infty}|x_1-x_2|\, , 
   \end{align*}
   where we used $|\psi|\leq 1$ and $|\nabla \delta_\Omega|\leq 1$ almost everywhere. Clearly, this is a contradiction if $|t|$ is sufficiently small.

   To prove the second claim, it suffices to observe that $|X(x)|\leq 1$:
   \begin{equation*}
    |X(x)| = 
    \eps_0^{-d}\biggl|\int_{\R^d}\varphi\Bigl(\frac{|x-y|}{\eps_0}\Bigr)\psi\Bigl(\frac{\delta_\Omega(y)}{\eta_0}\Bigr)\nabla\delta_\Omega(y)\, dy\biggr|
    \leq
    \eps_0^{-d}\int_{\R^d}\varphi\Bigl(\frac{|x-y|}{\eps_0}\Bigr)\, dy=1\, , 
   \end{equation*}
   since $|\psi|\leq 1$, $|\nabla \delta_\Omega|\leq 1$ almost everywhere and by the choice of normalization of $\varphi$.

   To prove the remaining statements, we argue as follows. Fix $\delta, \eta_0 >0$ and a point $x \in \partial\Omega$ such that the outward pointing unit normal to $\partial\Omega$ at $x$ is $(0, \ldots, 0, 1)$. Without loss of generality, we may assume that $x=0$. Provided $\eps_0>0$ is small enough (depending only on $\Omega$) the set $\partial\Omega \cap B_{2\eps_0}(0)$ can be parametrized as the graph of a function $f\in C^1(\R^{d-1})$, that is, 
   \begin{equation*}
    \partial\Omega\cap B_{2\eps_0}(0) = \{ (x', x_d) \in B_{2\eps_0}(0) : x_d=f(x')\}\, .
   \end{equation*}
   By the Heine--Cantor theorem and the compactness of $\partial\Omega$, there is a modulus of continuity $\omega\colon (0, \infty) \to \R$ (non-decreasing with $\lim_{\delta \to 0}\omega(\delta)=0$) independent of the choice of boundary point such that
   \begin{equation*}
    |\nabla f(x')-\nabla f(y')|\leq \omega(|x'-y'|)\quad \mbox{for all } x', y' \in \R^{d-1}\, .
   \end{equation*}
   Therefore, for any $\kappa>0$ there is an $\eps_0>0$ small enough (depending only on $\omega$) for which
   \begin{equation*}
    \partial\Omega \cap B_{\eps_0}(0)\subset \{x \in \R^d: |x_d| \leq \kappa |x|\} =: \mathcal{C}^0 \, .
   \end{equation*}
	Define the sets
   \begin{equation*}
    \mathcal{C}^+ := \{x \in \R^d: x_d > \kappa |x|\}\quad \mbox{and}\quad \mathcal{C}^- := \{x \in \R^d: x_d <- \kappa |x|\}\, .
   \end{equation*}
   Note that $\mathcal{C}^+\cap B_{\eps_0}(0)\subset \Omega^c$ and  $\mathcal{C}^-\cap B_{\eps_0}(0)\subset \Omega$.

   Assume that $\eta_0 > 2\eps_0$. Then at our boundary point $0$ we find
   \begin{align*}
    (0, \ldots, 0, 1) \cdot \Phi((1+\delta)t, 0))
    &\geq
    (1+\delta)t (0, \ldots, 0, 1)\cdot \eps_0^{-d}\int_{\R^d} \varphi\Bigl(\frac{|y|}{\eps_0}\Bigr)\nabla \delta_\Omega(y)\, dy\\
    &=
    -(1+\delta)t \eps_0^{-d-1}\int_{\R^d} \varphi'\Bigl(\frac{|y|}{\eps_0}\Bigr) \frac{y_d}{|y|}\delta_\Omega(y)\, dy\, .
   \end{align*}
   We claim that 
   \begin{equation}\label{eq: distance integral estimate}
    \eps_0^{-d-1}\int_{\R^d} \varphi'\Bigl(\frac{|y|}{\eps_0}\Bigr) \frac{y_d}{|y|}\delta_\Omega(y)\, dy \leq -1+ O(\kappa)\, , 
   \end{equation}
   where the implicit constant depends only on $d$ and the choice of $\varphi$.

   With the estimate~\eqref{eq: distance integral estimate} in hand, we see that $\Phi((1+\delta)t, \cdot)$ maps the origin into the set $A:=B_{(1+\delta)t}(0)\cap \{x\in \R^d: x_d \geq t(1+\delta)(1+ O(\kappa))$. If $\kappa$ is sufficiently small, we also have $\dist(A, \partial\Omega)>t$. Indeed, $\{x\in B_{(1+\delta)t}(0): \dist(x, \partial \Omega)\leq t\} \subseteq \{x=(x', x_d) \in B_{(1+\delta)t}(0): |x_d|\leq \sqrt{1+\kappa^2}t + \kappa |x'|\}$, which is disjoint from $A$, provided $\kappa$ is chosen sufficiently small. Thus, we have proved that $\Phi((1+\delta)t, 0)\in (\Omega+ tB)^c$. As the choice of boundary point was arbitrary, we conclude that $\Phi((1+\delta)t, \partial\Omega)\subset (\Omega+ tB)^c$. By the continuity of $\Phi$ and the fact that $\Phi$ acts as the identity in the bulk of $\Omega$, we have the desired inclusion $\Phi((1+\delta)t, \Omega)\supset \Omega +tB$. Similarly, the bound~\eqref{eq: distance integral estimate} together with $|X(x)|\leq 1$ implies that $1\geq \nu(x)\cdot X(x)\geq 1+O(\kappa)$ uniformly for all $x\in \partial\Omega$ and, therefore, $\|X-\nu\|_{L^\infty(\partial\Omega)}= O(\kappa)$ proving the final claim of the proposition.

   What remains to complete the proof of Proposition~\ref{prop: set inclusion} is to prove~\eqref{eq: distance integral estimate}. By splitting the integral, we can estimate
   \begin{align*}
    \int_{\R^d} \varphi'\Bigl(&\frac{|y|}{\eps_0}\Bigr) \frac{y_d}{|y|}\delta_\Omega(y)\, dy\\
    &=
    -\int_{\mathcal{C}^+} \Bigl|\varphi'\Bigl(\frac{|y|}{\eps_0}\Bigr)\Bigr| \frac{|y_d|}{|y|}\dist(y, \partial\Omega)\, dy
    -
    \int_{\mathcal{C}^-} \Bigl|\varphi'\Bigl(\frac{|y|}{\eps_0}\Bigr)\Bigr| \frac{|y_d|}{|y|}\dist(y, \partial\Omega)\, dy\\
    &\quad
    +
    \int_{\mathcal{C}^0} \varphi'\Bigl(\frac{|y|}{\eps_0}\Bigr) \frac{y_d}{|y|}\delta_\Omega(y)\, dy\\
    &\leq
    -\int_{\mathcal{C}^+} \Bigl|\varphi'\Bigl(\frac{|y|}{\eps_0}\Bigr)\Bigr| \frac{|y_d|}{|y|}\dist(y, \partial\mathcal{C}^+)\, dy
    -
    \int_{\mathcal{C}^-} \Bigl|\varphi'\Bigl(\frac{|y|}{\eps_0}\Bigr)\Bigr| \frac{|y_d|}{|y|}\dist(y, \partial\mathcal{C}^-)\, dy\\
    &\quad
    +
    O(\eps_0^{d+1}\kappa^3)\\
    &=
    -2\int_{\mathcal{C}^+} \Bigl|\varphi'\Bigl(\frac{|y|}{\eps_0}\Bigr)\Bigr| \frac{|y_d|}{|y|}\dist(y, \partial\mathcal{C}^+)\, dy
    +
    O(\eps_0^{d+1}\kappa^3)\, , 
   \end{align*}
   where we used $\varphi'\leq 0$, the definition of $\delta_\Omega$, $|\mathcal{C}^0\cap B_{\eps_0}(0)|\lesssim \kappa \eps_0^{d}$, and the fact that for $y \in \mathcal{C^0}\cap B_{\eps_0}(0)$ it holds that $\dist(y, \partial\Omega) \leq \kappa \eps_0$ and $|y_d|/|y|\leq \kappa$.

   By writing the remaining integral in spherical coordinates, we find
   \begin{align*}
  \int_{\mathcal{C}^+}&|\varphi'(|y|/\eps_0)|\frac{|y_d|}{|y|}\dist(y, \partial\mathcal{C}^+)\, dy\\
  &=
  |\S^{d-2}| \int_0^{\eps_0} \int_{0}^{\theta_0} |\varphi'(\eta/\eps_0)| \cos(\theta)\eta (\sin(\theta), \cos(\theta))\cdot (- \kappa, \sqrt{1-\kappa^2})\sin(\theta)^{d-2}\eta^{d-1}\, d\theta d\eta\\
  &=
  |\S^{d-2}| \int_0^{\eps_0} \int_{0}^{\pi/2} |\varphi'(\eta/\eps_0)| \cos(\theta)^2\sin(\theta)^{d-2}\eta^{d}\, d\theta d\eta\\
  &\quad
  - 
  |\S^{d-2}| \int_0^{\eps_0} \int_{\theta_0}^{\pi/2} |\varphi'(\eta/\eps_0)| \cos(\theta)^2\sin(\theta)^{d-2}\eta^{d}\, d\theta d\eta\\
  &\quad
  + |\S^{d-2}|( \sqrt{1-\kappa^2}-1) \int_0^{\eps_0} \int_{0}^{\theta_0} |\varphi'(\eta/\eps_0)| \cos(\theta)^2\sin(\theta)^{d-2}\eta^{d}\, d\theta d\eta\\
  &\quad
  -
  |\S^{d-2}| \kappa\int_0^{\eps_0} \int_{0}^{\theta_0} |\varphi'(\eta/\eps_0)| \cos(\theta)\sin(\theta)^{d-1}\eta^{d}\, d\theta d\eta\\
  &=
  \frac{|\S^{d-1}|}{2d}\int_0^{\eps_0} |\varphi'(\eta/\eps_0)|\eta^d \, d\eta + O(\kappa\eps_0^{d+1})\\
  &=
  \frac{1}{2}\eps_0^{d+1} + O(\kappa\eps_0^{d+1})\, , 
\end{align*}
where we set
\begin{equation*}
  \theta_0 := \arccos(\kappa)\, .
\end{equation*}
This proves the estimate~\eqref{eq: distance integral estimate} and thus completes the proof of Proposition~\ref{prop: set inclusion}.
\end{proof}

\subsection{Proof of Theorem~\ref{thm: Hadamard Minkowski sum}}
By Proposition~\ref{prop: set inclusion}, for any $\delta, \eta_0>0$ there is an $\eps_0>0$ so that for $t>0$ small enough $\Phi_\Omega(t, \Omega) \subseteq \Omega + tB \subseteq \Phi_\Omega((1+\delta)t, \Omega)$. By the monotonicity of $\lambda_q$ under set inclusions, we conclude that
\begin{align*}
  \frac{\lambda_q(\Phi_\Omega(t, \Omega))-\lambda_q(\Omega)}{t} 
  \leq 
  \frac{\lambda_q(\Omega + tB)-\lambda_q(\Omega)}{t}
  \leq 
  \frac{\lambda_q(\Phi_\Omega((1+\delta)t, \Omega))-\lambda_q(\Omega)}{t}\, .
\end{align*}

By Theorem~\ref{thm: variation}, 
\begin{align*}
  \lim_{t\to 0^\limplus}\frac{\lambda_q(\Phi_\Omega(t, \Omega))-\lambda_q(\Omega)}{t}  &=  -\int_{\partial \Omega} \Bigl( \frac{\partial u_{q, \Omega}}{\partial \nu}\Bigr)^2 X\cdot \nu \, d\Haus^{d-1}(x)\, , \\
  \lim_{t\to 0^\limplus}\frac{\lambda_q(\Phi_\Omega((1+\delta)t, \Omega))-\lambda_q(\Omega)}{t}  &=  -(1+\delta)\int_{\partial \Omega} \Bigl( \frac{\partial u_{q, \Omega}}{\partial \nu}\Bigr)^2 X\cdot \nu \, d\Haus^{d-1}(x)\, .
 \end{align*} 

 By sending $\eps_0\to 0$, dominated convergence along with~\ref{it: normal convergence} of Proposition~\ref{prop: set inclusion}, we deduce
 \begin{align*}
   -\int_{\partial \Omega} \!\Bigl( \frac{\partial u_{q, \Omega}}{\partial \nu}\Bigr)^2 d\Haus^{d-1}(x)
   &\leq \liminf_{t \to 0^\limplus}\frac{\lambda_q(\Omega + tB)-\lambda_q(\Omega)}{t}\\
   &\leq \limsup_{t \to 0^\limplus} \frac{\lambda_q(\Omega + tB)-\lambda_q(\Omega)}{t} \leq -(1+\delta)\int_{\partial \Omega} \!\Bigl( \frac{\partial u_{q, \Omega}}{\partial \nu}\Bigr)^2 d\Haus^{d-1}(x)\, .
 \end{align*}
 Consequently, since $\delta>0$ is arbitrary, this completes the proof of the theorem. \qed


\section{Proof of Theorem~\ref{thm: semi-linear eigenvalue}}

In this section, we prove our main result, Theorem~\ref{thm: semi-linear eigenvalue}. We will do this in several steps. In the first step, we argue that it suffices to prove the theorem for $\Omega$ open and connected. In the second step, we prove that, under the additional assumption that $\Omega$ has $C^1$ boundary, the inequality in Theorem~\ref{thm: semi-linear eigenvalue} follows by combining Theorems~\ref{thm: BM lambdaq} and~\ref{thm: Hadamard Minkowski sum}. Finally, in a third step, we show that the inequality for $\Omega$ with Lipschitz boundary can be deduced from the more regular case by a fairly standard approximation argument.


\subsection{Reduction to connected sets} 

In this subsection, we prove that, if $\Omega$ is an open set of finite measure and if $\Omega = \cup_{j\in J} \Omega_j$ with open, connected and pairwise disjoint $\Omega_j$, then inequality of Theorem~\ref{thm: semi-linear eigenvalue} for $\Omega$ follows from the result applied to each individual $\Omega_j$ separately. In particular, it suffices to prove Theorem~\ref{thm: semi-linear eigenvalue} under the additional assumption that $\Omega$ is connected.

\medskip

{\noindent \bf Case 1} ($q=2$){\bf:} Let $\Omega = \cup_{j\in J} \Omega_J$ as above and assume that the statement of the theorem with $q=2$ holds for $\Omega_j$ with $j\in J$. Without loss of generality we assume that $\lambda_2(\Omega_{j})\leq \lambda_2(\Omega_{j+1})$ for all $j$. 

If $\lambda_2(\Omega_1)<\lambda_2(\Omega_2)$, then $\lambda_2(\Omega)=\lambda_2(\Omega_1)$ and any associated eigenfunction $u_{2, \Omega}$ is an eigenfunction on $\Omega_1$ (extended by zero to $\Omega\setminus \Omega_1$). Therefore, the statement of the theorem for $\Omega$ follows immediately from the validity of the theorem for $\Omega_1$.

If $\lambda_2(\Omega_1)=\lambda_2(\Omega_m)$ for some maximal $m\geq 2$ (note that such an $m$ exists by the argument in the proof of Lemma \ref{uniquevar}), then $\lambda_2(\Omega)=\lambda_2(\Omega_1)=\lambda_2(\Omega_m)$ and any eigenfunction $u_{2, \Omega}$ is a linear combination of eigenfunctions on $\{\Omega_j\}_{j=1}^m$  extended by zero. That is, if $u_{2, \Omega}$ is an $L^2$-normalized eigenfunction associated to $\lambda_2(\Omega)$, then there are $\{a_j\}_{j=1}^m$ with $\sum_{j=1}^m a_j^2=1$ such that $u_{2, \Omega}= \sum_{j=1}^m a_j u_{2, \Omega_j}$, where, for each $j$, $u_{2, \Omega_j}$ is an $L^2$-normalized eigenfunction associated to $\lambda_2(\Omega_j)$. Since the $u_{2, \Omega_j}$ have disjoint support for different $j$, from the inequality of Theorem~\ref{thm: semi-linear eigenvalue} applied to each $\Omega_j$ separately we deduce that 
\begin{align*}
  \int_{\partial\Omega} \Bigl(\frac{\partial u_{2, \Omega}}{\partial \nu}\Bigr)^2 \, d\Haus^{d-1}(x) 
  &\geq 
  \sum_{j=1}^m a_j^2\int_{\partial\Omega_{j}} \Bigl(\frac{\partial u_{2, \Omega_j}}{\partial \nu}\Bigr)^2 \, d\Haus^{d-1}(x)\\
  &\geq
  2\sum_{j=1}^m a_j^2 \frac{\lambda_2(\Omega_j)^{3/2}}{\lambda_2(B)^{1/2}}\\
  &= 2 \frac{\lambda_2(\Omega)^{3/2}}{\lambda_2(B)^{1/2}}\, , 
\end{align*}
which is the desired inequality for $\Omega$.

\medskip

{\noindent \bf Case 2} ($1\leq q<2$){\bf:}

Fix $1\leq q <2$, let $\Omega = \cup_{j\in J} \Omega_j$ as above, and assume that the statement of the theorem holds for each of the sets $\Omega_j$ with $j\in J$. 

In this case, the normalized minimizers $u_{q, \Omega}, u_{q, \Omega_j}$ are all unique. Moreover, by~\eqref{eq: spin formula0} and~\eqref{eq: spin formula}, 
\begin{align*}
  \lambda_q(\Omega)^{-\frac{q}{2-q}}&=\sum_{j\in J}\lambda_q(\Omega_j)^{-\frac{q}{2-q}}\, , \ \mbox{and}\\
  u_{q, \Omega}(x) &= \sum_{j\in J}\Bigl(\frac{\lambda_q(\Omega)}{\lambda_q(\Omega_j)}\Bigr)^{\frac{1}{2-q}}u_{q, \Omega_j}(x) \, .
\end{align*} 

By the disjointness of the supports of the $u_{q, \Omega_j}$, we observe that
\begin{equation}\label{eq: subadditivity}
\begin{aligned}
  \int_{\partial\Omega} \Bigl( \frac{\partial u_{q, \Omega}}{\partial \nu}\Bigr)^2 \, d\Haus^{d-1}(x)
  &=
  \sum_{j\in J} \Bigl(\frac{\lambda_q(\Omega)}{\lambda_q(\Omega_j)}\Bigr)^{\frac{2}{2-q}}\int_{\partial\Omega_j} \Bigl(\frac{\partial u_{q, \Omega_j}}{\partial \nu}\Bigr)^2 \, d\Haus^{d-1}(x)\\
  &\geq
  \frac{\lambda_q(\Omega)^{\frac{2}{2-q}}}{\alpha_q \lambda_q(B)^{\alpha_q}} \sum_{j\in J} \Bigl(\lambda_q(\Omega_j)^{-\frac{q}{2-q}}\Bigr)^{-\frac{2-q}{q}(1+\alpha_q-\frac{2}{2-q})}\, .
\end{aligned}
\end{equation}
By the definition of $\alpha_q$ and since $1\leq q<2$, we see that
\begin{equation*}
  -\frac{2-q}{q}\Bigl(1+\alpha_q-\frac{2}{2-q}\Bigr) = 1- \frac{2-q}{2q+d(2-q)} \in \Bigl[1-\frac{1}{d+2}, 1\Bigr)\, .
\end{equation*}
Consequently,~\eqref{eq: subadditivity}, the subadditivity of $x \mapsto x^{\alpha}$ for $0<\alpha \leq 1$, and~\eqref{eq: spin formula0}, \eqref{eq: spin formula} imply that
\begin{align*}
  \int_{\partial\Omega} \Bigl( \frac{\partial u_{q, \Omega}}{\partial \nu}\Bigr)^2 \, d\Haus^{d-1}(x)
  &\geq
  \frac{\lambda_q(\Omega)^{\frac{2}{2-q}}}{\alpha_q \lambda_q(B)^{\alpha_q}} \sum_{j\in J} \Bigl(\lambda_q(\Omega_i)^{-\frac{q}{2-q}}\Bigr)^{-\frac{2-q}{q}(1+\alpha_q-\frac{2}{2-q})}\\
  &\geq
  \frac{\lambda_q(\Omega)^{\frac{2}{2-q}}}{\alpha_q \lambda_q(B)^{\alpha_q}} \biggl(\sum_{j\in J} \lambda_q(\Omega_i)^{-\frac{q}{2-q}}\biggr)^{-\frac{2-q}{q}(1+\alpha_q-\frac{2}{2-q})}\\
  &=
  \frac{\lambda_q(\Omega)^{1+\alpha_q}}{\alpha_q \lambda_q(B)^{\alpha_q}}\, , 
\end{align*}
which is the desired inequality.


\subsection{Proof of Theorem~\ref{thm: semi-linear eigenvalue} for \texorpdfstring{$C^1$}{C1} sets}

We are now ready to prove Theorem~\ref{thm: semi-linear eigenvalue} under the assumption that the boundary of $\Omega$ is $C^1$.

Recall that $B$ denotes the unit ball. Since 
\begin{equation*}
  \Omega + t B = (1+t)\Bigl[\Bigl(1- \frac{t}{1+t}\Bigr)\Omega + \frac{t}{1+t}B\Bigr] \, , 
\end{equation*}
the homogeneity of $\lambda_q$ and the Brunn--Minkowski inequality of Theorem~\ref{thm: BM lambdaq} imply that
\begin{align*}
  \lambda_q(\Omega+t B) &= (1+t)^{-1/\alpha_q}\Bigl[\lambda_q\Bigl(\Bigl(1- \frac{t}{1+t}\Bigr)\Omega + \frac{t}{1+t}B\Bigr)^{-\alpha_q}\Bigr]^{-1/\alpha_q}\\
  &\leq 
  (1+t)^{-1/\alpha_q}\Bigl[\Bigl(1- \frac{t}{1+t}\Bigr)\lambda_q(\Omega)^{-\alpha_q} + \frac{t}{1+t}\lambda_q(B)^{-\alpha_q}\Bigr]^{-1/\alpha_q}\\
  &= 
  \bigl(\lambda_q(\Omega)^{-\alpha_q} + t\lambda_q(B)^{-\alpha_q}\bigr)^{-1/\alpha_q}\, .
\end{align*}
When combined with Theorem~\ref{thm: Hadamard Minkowski sum}, we find
\begin{equation*}
\begin{aligned}
  -\int_{\partial \Omega} \Bigl( \frac{\partial u_{q, \Omega}}{\partial \nu}\Bigr)^2\, d\Haus^{d-1}(x) &=\lim_{t \to 0} \frac{\lambda_q(\Omega + tB)-\lambda_q(\Omega)}{t}\\
  &\leq 
  \lim_{t \to 0} \frac{\bigl(\lambda_q(\Omega)^{-\alpha_q} + t\lambda_q(B)^{-\alpha_q}\bigr)^{-1/\alpha_q}-\lambda_q(\Omega)}{t} \\
  &= - \frac{\lambda_q(\Omega)^{1+\alpha_q}}{\alpha_q\lambda_q(B)^{\alpha_q}}\, , 
\end{aligned}
\end{equation*}
which completes the proof of the inequality in Theorem~\ref{thm: semi-linear eigenvalue} under the assumption that $\Omega$ has $C^1$ boundary. 

If $\Omega$ is a ball of radius $r$, then $\Omega + tB$ is a ball of radius $r+t$. Therefore, by homogeneity of $\lambda_q$, equality holds for each $t$ in the above application of the Brunn--Minkowski inequality. Consequently, when $\Omega$ is a ball, equality holds in the inequality of Theorem~\ref{thm: semi-linear eigenvalue}.

There is another way to deduce that equality holds for balls. Namely, for any bounded open set $\Omega$ with Lipschitz boundary one has the Rellich--Pohozaev identity, 
$$
\int_{\partial\Omega} \Bigl(\frac{\partial u_{q, \Omega}}{\partial \nu}\Bigr)^2x \cdot \nu \, d\Haus^{d-1}(x) = \frac{\lambda_q(\Omega)}{\alpha_q} \, ;
$$
see, e.g., \cite[Proposition 2.9]{BrascoFranzina_Semilinear_overview_2020}. In particular, when $\Omega$ is a ball, centered at the origin, then $x\cdot\nu$ is equal to the radius of the ball and the Rellich--Pohozaev identity reduces to the equality case in Theorem~\ref{thm: semi-linear eigenvalue}.


\subsection{Approximation of Lipschitz sets from the inside by smooth sets}
\label{sec: approx Lipschitz}

Fix $1\leq q \leq 2$ and let $\Omega \subset \R^d$ be open bounded and connected with Lipschitz boundary.

Our goal is to prove the inequality in Theorem~\ref{thm: semi-linear eigenvalue} by approximating $\Omega$ from the inside by smooth sets and proving that the involved quantities converge under this approximation.

Let $\{\Omega_j\}_{j\geq 1}$ be a sequence of open sets with $C^\infty$-regular boundary such that:
\begin{enumerate}[label=(\roman*)]
  \item\label{it: Lip set approx inclusion} $\Omega_j \subset \Omega_{j+1}$ for all $j\geq 1$ and $\Omega = \cup_{j\geq 1} \Omega_j$.

  \item\label{it: Lip set approx homeo} There are homeomorphisms $B_j\colon \partial\Omega \to \partial\Omega_j$ such that $$\sup_{x\in \partial\Omega}|x-B_j(x)|= o_{j\to \infty}(1)$$ and for all $x\in \partial\Omega$ the $B_j(x)$ convergences to $x$ non-tangentially.

  \item\label{it: Lip set approx Jacobian} There is $\delta \in (0, 1)$ and functions $b_j\colon \partial\Omega \to (\delta, 1/\delta)$ such that, for all measurable $\omega \subset \partial\Omega$, 
  \begin{equation*}
     \int_\omega b_j(x)\, d\Haus^{d-1}(x) = \int_{B_j(\omega)}\, d\Haus^{d-1}(x)
   \end{equation*} 
   and $b_j \to 1$ pointwise almost everywhere and in $L^p(\partial\Omega)$, for all $1\leq p <\infty$, 

   \item\label{it: Lip set approx normal convergence} Let $\nu_{j}(x)$ denote the outward pointing unit normal to $\partial\Omega_j$ at $x\in \partial\Omega_j$. The function $\partial \Omega \ni x\mapsto \nu_{j}(B_j(x))$ converges to $\nu(x)$ pointwise almost everywhere and in $L^p(\partial\Omega)$, for all $1\leq p <\infty$. The corresponding statement holds also for the locally defined tangent vectors.
\end{enumerate}
The existence of a sequence $\{\Omega_j\}_{j\geq 1}$ satisfying~\ref{it: Lip set approx inclusion}--\ref{it: Lip set approx normal convergence} is the content of~\cite[Theorem 1.12]{verchota_layer_1984}

Combining~\ref{it: Lip set approx inclusion} and~\ref{it: Lip set approx homeo} we deduce that $\Omega_j \to \Omega$ with respect to the Hausdorff distance. By Lemma~\ref{lem: Hausdorff interior continuity}, it holds that $\lim_{j \to \infty}\lambda_q(\Omega_j) = \lambda_q(\Omega)$. Moreover, Lemma~\ref{lem: strong eigenfunction conv} implies that $u_{q, \Omega_j}$ converges to $u_{q, \Omega}$ strongly in $H_0^1(\Omega)$. 

By the results of Jerison--Kenig~\cite{JerisonKenig_BAMS81} and Verchota~\cite{verchota_layer_1984}, the non-tangential maximal function of $\nabla u_{q, \Omega}$ belongs to $L^2(\partial\Omega)$. Therefore, by the dominated convergence theorem together with properties~\ref{it: Lip set approx Jacobian} and~\ref{it: Lip set approx normal convergence}, 
\begin{equation}\label{eq: normal derivative on inner boundary}
  \int_{\partial\Omega} \Bigl( \frac{\partial u_{q, \Omega}}{\partial \nu}\Bigr)^2 \, d\Haus^{d-1}(x) = \lim_{j \to \infty} \int_{\partial \Omega_j} \Bigl(\frac{\partial u_{q, \Omega}}{\partial \nu_j}\Bigr)^2\, d\Haus^{d-1}(x)\, .
\end{equation}

We claim that
\begin{equation}\label{eq: convergence of normal derivatives}
  \lim_{j\to \infty}\biggl\|\frac{\partial (u_{q, \Omega}-u_{q, \Omega_j})}{\partial\nu_j}\biggr\|_{L^2(\partial\Omega_j)} = 0\, .
\end{equation}
Before proving this statement, let us show how it implies Theorem~\ref{thm: semi-linear eigenvalue}.

By~\eqref{eq: normal derivative on inner boundary}, the Cauchy--Schwarz inequality,~\eqref{eq: convergence of normal derivatives}, and Theorem~\ref{thm: semi-linear eigenvalue} applied for the smooth sets~$\Omega_j$, 
\begin{align*}
  \int_{\partial\Omega} \Bigl(\frac{\partial u_{q, \Omega}}{\partial\nu}\Bigr)^2 \, d\Haus^{d-1}(x)
  &=
  \lim_{j\to \infty}\int_{\partial \Omega_j} \Bigl(\frac{\partial u_{q, \Omega}}{\partial \nu_j}\Bigr)^2\, d\Haus^{d-1}(x)\\
  &=
  \lim_{j\to \infty}\int_{\partial \Omega_j} \Bigl(\frac{\partial u_{q, \Omega_j}}{\partial \nu_j}\Bigr)^2\, d\Haus^{d-1}(x)\\
  &\geq
  \limsup_{j\to \infty} \frac{\lambda_q(\Omega_j)^{1+\alpha_q}}{\alpha_q \lambda_q(B)^{\alpha_q}}\\
  & = \frac{\lambda_q(\Omega)^{1+\alpha_q}}{\alpha_q \lambda_q(B)^{\alpha_q}}\, .
\end{align*}
This is the inequality claimed in Theorem~\ref{thm: semi-linear eigenvalue} for $\Omega$. Therefore, all that remains to complete the proof of the theorem is to verify~\eqref{eq: convergence of normal derivatives}. 

As in Section \ref{sec: Prelminaries}, let $\Gamma$ be the Newtonian potential in $\R^d$. We set
\begin{equation*}
  w_j := -\Gamma*\Bigl[\lambda_q(\Omega)u_{q, \Omega}^{q-1}-\lambda_q(\Omega_j)u_{q, \Omega_j}^{q-1}\Bigr]\, , 
\end{equation*}
and let $v_j$ be defined by
\begin{equation*}
  u_{q, \Omega}-u_{q, \Omega_j} = w_j + v_j\, .
\end{equation*}
We will show that the analogue of~\eqref{eq: convergence of normal derivatives} holds separately for $w_j$ and for $v_j$.

We begin with $w_j$ and use the fact that $\{u_{q, \Omega_j}\}_{j\geq 1}$ is bounded in $L^\infty(\Omega)$. This follows from inequality \eqref{eq:linfty} together with domain monotonicity of $\lambda_q$, using $\Omega_j \subset \Omega$. Consequently, $w_j$ are uniformly bounded in $C^{1, \alpha}_{\mathrm{loc}}(\R^d)$, for any $\alpha<1$ (see, for instance,~\cite[Theorem~10.2]{LiebLoss}). By the Arzel\`a--Ascoli theorem and passing to a subsequence, we can assume that $w_j$ converges in $C^{1, \alpha}_{\mathrm{loc}}(\R^d)$, for all $\alpha<1$, to some limit $w$. Since $\lambda_q(\Omega_j) \to \lambda_q(\Omega)$ and $u_{q, \Omega_j} \to u_{q, \Omega}$ in $H_0^1(\Omega)$, an application of Young's inequality (as in the proof of~\cite[Theorem 10.2]{LiebLoss}) implies that $w_j\to 0$ in $L^1(\Omega)$, consequently $w=0$. By convergence in $C^{1, \alpha}_{\mathrm{loc}}(\R^d)$ and properties~\ref{it: Lip set approx Jacobian}-\ref{it: Lip set approx normal convergence}, we deduce that
\begin{equation*}
  \lim_{j\to \infty} \Bigl\|\frac{\partial w_j}{\partial \nu_j}\Bigr\|_{L^2(\partial\Omega_j)} =0\, .
\end{equation*}
This is the analogue of~\eqref{eq: convergence of normal derivatives} for $w_j$.

We now turn to $v_j$. By construction, they solve
\begin{equation*}
  \begin{cases}
    -\Delta v_j =0 \quad & \mbox{in }\Omega_j\, , \\
    v_j =  u_{q, \Omega}- w_j & \mbox{on }\partial\Omega_j\, .
  \end{cases}
\end{equation*}
Write $v_j := v_{j, 1}-v_{j, 2}$ with
\begin{align*}
  \begin{cases}
    -\Delta v_{j, 1} =0 \quad & \mbox{in }\Omega_j\, , \\
    v_{j, 1} =  u_{q, \Omega} & \mbox{on }\partial\Omega_j\, , 
  \end{cases}\qquad \mbox{resp.}\qquad 
  \begin{cases}
    -\Delta v_{j, 2} =0 \quad & \mbox{in }\Omega_j\, , \\
    v_{j, 2} =  w_j & \mbox{on }\partial\Omega_j\, .
  \end{cases}
\end{align*}
We will now use the results in~\cite{verchota_layer_1984} to argue that $\|\frac{\partial v_{j, i}}{\partial\nu_j}\|_{L^2(\partial\Omega_j)}\to 0$ for $i=1, 2$. This gives the analogue of~\eqref{eq: convergence of normal derivatives} for $v_j$ and therefore completes the proof.

By arguing as in the proof of~\cite[Theorem 2.1]{verchota_layer_1984} (see also~\cite[Theorem 2]{JerisonKenig_BAMS81}) we can bound
\begin{equation*}
  \Bigl\|\frac{\partial v_{j, i}}{\partial \nu_j}\Bigl\|_{L^2(\partial \Omega_j)}\leq C \|\nabla_t v_{j, i}\|_{L^2(\partial \Omega_j)}
\end{equation*}
with a constant $C$ that is independent of $j$. Here, $\nabla_t$ denotes the tangential derivative on $\partial\Omega_j$. Since $\partial\Omega_j$ is smooth and $u_{q, \Omega}, w_j \in C^{1, \alpha}(\partial\Omega_j)$, Schauder theory implies that the boundary data are achieved in the classical sense, so $\nabla_t v_{j, 1}= \nabla_t u_{q, \Omega}$ and $\nabla_t v_{j, 2} = \nabla_t w_j$ on $\partial\Omega_j$. Since $w_j$ converges to zero in $C^{1, \alpha}_{loc}(\R^d)$, we have $\|\nabla_t w_j\|_{L^2(\partial\Omega_j)}\to 0$ and, consequently, $\|\frac{\partial v_{j, 2}}{\partial \nu_j}\|_{L^2(\partial\Omega_j)}\to 0$.

The fact that $\|\nabla_t u_{q, \Omega}\|_{L^2(\partial\Omega_j)}\to 0$ follows from~\ref{it: Lip set approx normal convergence}, the fact that $\nabla u_{q, \Omega}$ has a non-tangential limit which is normal to the boundary almost everywhere on $\partial\Omega$, the fact that the non-tangential maximal function of $\nabla u_{q, \Omega}$ belongs to $L^2(\partial\Omega)$, and the dominated convergence theorem. Consequently, $\|\frac{\partial v_{j, 1}}{\partial \nu_j}\|_{L^2(\partial\Omega_j)}\to 0$. This concludes the proof.

\appendix

\section{On equality in Theorem~\ref{thm: BM lambdaq}}
\label{Appendix}

In this appendix we characterize the equality cases in the Brunn--Minkowski-type inequality of Theorem~\ref{thm: BM lambdaq} in the case $1\leq q<2$. That is, we will extend Colesanti's result~\cite{colesanti_brunnminkowski_2005} for bounded open sets with $C^2$ boundary to arbitrary open sets.

Let $\Omega_0, \Omega_1$ be non-empty and open. For $t\in (0, 1)$ set
$$
\Omega_t := (1-t)\Omega_0+t\Omega_1 \, .
$$

As shown at the beginning of Section \ref{sec: BM inequality}, if $\min\{\lambda_q(\Omega_0), \lambda_q(\Omega_1)\}=0$, then for all $t\in (0, 1)$ one has $\lambda_q(\Omega_t)=0$ and, consequently, equality holds in the inequality. Thus, the only case where one can hope to characterize $\Omega_0, \Omega_1$ yielding equality is when neither $\lambda_q(\Omega_0)$ nor $\lambda_q(\Omega_1)$ is zero. If $\min\{\lambda_q(\Omega_0), \lambda_q(\Omega_1)\}>0$ but $\lambda_q(\Omega_t)=0$, then clearly equality does not hold in the inequality. To characterize equality cases we can thus, without loss of generality, assume that $\lambda_q(\Omega_i)>0$ for $i=0, 1, t$. 

\begin{lemma}\label{eq:equalitybdd}
	Fix $q\in [1, 2)$ and $t\in (0, 1)$. Let $\Omega_0, \Omega_1\subset \R^d$ be non-empty open sets. If $\lambda_q(\Omega_t)>0$, then the sets $\Omega_0$, $\Omega_1$, and $\Omega_t$ are bounded.
\end{lemma}

\begin{proof}
	If $\Omega_0$ and $\Omega_1$ are bounded, then so is $\Omega_t$. Therefore it suffices to show that, if $\Omega_0$ or $\Omega_1$ is unbounded, then $\lambda_q(\Omega_t)=0$. We argue by contradiction. Assume for definiteness that $\Omega_0$ is unbounded.

	Since $\Omega_0$ is unbounded, we can find a sequence $\{x_n\}_{n\geq 1} \subset \Omega_0$ such that $|x_n|+t/(1-t) \leq |x_{n+1}|$ for each $n\geq 1$. Indeed, the sequence can be constructed by induction: pick an arbitrary point in $\Omega_0$ as $x_1$. Given $\{x_n\}_{n=1}^N$ the set $\Omega_0 \setminus B_{|x_N|+t/(1-t)}(0)$ is non-empty, since otherwise $\Omega_0$ would be bounded, and $x_{N+1}$ can be chosen arbitrarily in this set. 
	
	Fix $y\in\Omega_1$. Since $\Omega_1$ is open, there is an $\eps \in (0, 1]$ such that $B_\eps(y)\subset \Omega_1$. The set $\Omega_t$ contains $\cup_{n\geq 1} B_{t\eps}((1-t)x_n+ty)$. By construction, $|((1-t)x_n+ty)-((1-t)x_m+ty)|\geq t\geq t\eps$ for each $n\neq m$, and thus the balls in the union are disjoint. Using the monotonicity of $\lambda_q$ under set inclusions and~\eqref{eq: spin formula0}, we conclude that $\lambda_q(\Omega_t)=0$.
 \end{proof}

With the above facts in hand we are ready to prove the following result, which generalizes~\cite[Theorem 20]{colesanti_brunnminkowski_2005}.

\begin{lemma}\label{lem: pointwise ineq}
	Fix $q\in [1, 2)$ and $t\in (0, 1)$. Let $\Omega_0, \Omega_1\subset \R^d$ be non-empty open sets. If $\lambda_q(\Omega_i)>0$ for $i=0, 1, t$, then, for all $(x, y)\in \Omega_0\times \Omega_1$, 
	\begin{equation*}
		\lambda_q(\Omega_t)^{-1/2}u_{q, \Omega_t}((1-t)x+ty)^{\frac{2-q}{2}} \geq (1-t)\lambda_q(\Omega_0)^{-1/2}u_{q, \Omega_0}(x)^{\frac{2-q}{2}}+t \lambda_q(\Omega_1)^{-1/2}u_{q, \Omega_1}(y)^{\frac{2-q}{2}} \, .
	\end{equation*}
\end{lemma}

\begin{proof}[Proof of Lemma~\ref{lem: pointwise ineq}]
	If $\Omega_0, \Omega_1$ are bounded sets with $C^2$ boundary (or convex), then the claimed inequality is shown in the proof of~\cite[Theorem 20]{colesanti_brunnminkowski_2005}; see also~\cite[Remark~22]{colesanti_brunnminkowski_2005}.

	If $\Omega_0, \Omega_1$ are as the lemma, then, by the previous lemma, $\Omega_0, \Omega_1$ are bounded. Consequently, there are interior exhaustions $\{\Omega_i^j\}_{j\geq 1}$ for $i=0, 1$ such that $\Omega_i^j$ is bounded and has $C^2$-regular boundary, $\Omega_i^j\subset \Omega_i^{j+1}\subset \Omega_i$ for all $j\geq 1$, $\cup_{j\geq 1}\Omega_i^j = \Omega_i$ and $\Omega_i^j \to \Omega_i$ with respect to the Hausdorff distance. By Lemma~\ref{lem: Hausdorff interior continuity}, $\lambda_q(\Omega_i^j)\to \lambda_q(\Omega_i)$. Therefore, Lemma~\ref{lem: strong eigenfunction conv} implies that $u_{q, \Omega_i^j} \to u_{q, \Omega_i}$ in $H_0^1(\Omega_i)$.

	For each $j\geq 1$, define $\Omega_t^j := (1-t)\Omega_0^j + t \Omega_1^j$. Using the properties of $\Omega_i^j$ and the fact that $\Omega_0$ and $\Omega_1$ are bounded, we easily see that $\Omega_t^j$ satisfies the assumptions of Lemma~\ref{lem: Hausdorff interior continuity}. Hence Lemmas~\ref{lem: Hausdorff interior continuity} and~\ref{lem: strong eigenfunction conv} imply that $\lambda_q(\Omega_t^j)\to \lambda_q(\Omega_t)$ and $u_{q, \Omega_t^j}\to u_{q, \Omega_t}$ in $H^1_0(\Omega_t)$.

	By passing to a subsequence in $j$, we may assume that $u_{q, \Omega_0^j}(x)\to u_{q, \Omega_0}(x)$ for almost every $x\in \Omega_0$. Let $\{x_n\}_{n \geq 1}$ be countable set of such points that is dense in $\Omega_0$. For fixed $x \in \Omega_0$ we can pass to a further subsequence so that $u_{q, \Omega_t^j}((1-t)x+ty) \to u_{q, \Omega_t}((1-t)x+ty)$ and $u_{q, \Omega_1^j}(y) \to u_{q, \Omega_1}(y)$ for almost every $y \in \Omega_1$. By a diagonal argument, we can pass to a subsequence in $j$ so that, for each $x_n$ and almost every $y \in \Omega_1$
	\begin{equation*}
		u_{q, \Omega_t^j}((1-t)x_n +t y) \to u_{q, \Omega_t}((1-t)x_n +t y)\quad \mbox{and} \quad u_{q, \Omega_1^j}(y) \to u_{q, \Omega_1}(y) \, .
	\end{equation*}
	Since the intersection of countably many sets of full measure is again a set of full measure, the convergence above holds almost everywhere in $\Omega_1$ simultaneously for all $n$. 

By applying the pointwise inequality of Colesanti to the sets in the exhaustion and by passing to the limit in $j$, we conclude that, for each $x_n$ and almost all $y \in \Omega_1$, 
\begin{equation*}
	 \lambda_q(\Omega_t)^{-1/2}u_{q, \Omega_t}((1-t)x_n+ty)^{\frac{2-q}{2}}\geq (1-t)\lambda_q(\Omega_0)^{-1/2}u_{q, \Omega_0}(x_n)^{\frac{2-q}{2}}+t\lambda_q(\Omega_1)^{-1/2}u_{q, \Omega_q}(y)^{\frac{2-q}{2}}\, .
\end{equation*}
By elliptic regularity (Lemma \ref{ellipticreg}), both sides of this inequality are continuous functions of $y\in \Omega_1$, and hence it extends to all $y \in \Omega_1$. Similarly, for fixed $y\in \Omega_1$ the continuity of both sides of the inequality as a function of $x_n$ implies that it extends also to all $x\in \Omega_0$. This completes the proof of Lemma~\ref{lem: pointwise ineq}.  
\end{proof}

\begin{proof}[Proof of Theorem~\ref{thm: BM lambdaq}. Equality cases for $1\leq q<2$]
	Let $\Omega_0$ and $\Omega_1$ be non-empty open sets with $\lambda_q(\Omega_i)>0$ for $i=0, 1$ and such that equality holds in \eqref{eq: BM ineq thm}. We continue to use the notation $\Omega_t=(1-t)\Omega_0+t\Omega_1$ and recall from the discussion at the beginning of this appendix that $\lambda_q(\Omega_t)>0$. Thus, by Lemma \ref{eq:equalitybdd}, $\Omega_i$, $i=0, 1, t$, are all bounded.
	
	By a simple rescaling argument as in the proof of~\cite[Theorem~11]{colesanti_brunnminkowski_2005}, we may assume that equality holds in the multiplicative inequality 
	\begin{equation*}
		\lambda_q(\Omega_t)^{\frac{q}{q-2}}\geq \lambda_q(\Omega_0)^{(1-t)\frac{q}{q-2}}\lambda_q(\Omega_1)^{t\frac{q}{q-2}}.
	\end{equation*}
	We set
	$$
	f:=\lambda_q(\Omega_0)^{\frac{q}{q-2}}u_{q, \Omega_0}^q \, , 
	\qquad
	g:=\lambda_q(\Omega_1)^{\frac{q}{q-2}}u_{q, \Omega_1}^q \, , 
	\qquad
	h=\lambda_q(\Omega_t)^{\frac{q}{q-2}}u_{q, \Omega_t}^q \, .
	$$
	By Lemma~\ref{lem: pointwise ineq}, 
	\begin{equation*}
		h((1-t)x+ty)\geq \Bigl[(1-t)f(x)^r+tg(y)^r\Bigr]^{1/r} \quad \mbox{for all }(x, y)\in \Omega_0\times \Omega_1 \mbox{ and } r := \frac{2-q}{2q}\, .
	\end{equation*}
	Thus, by the arithmetic-geometric means inequality, 
	\begin{equation*}
		h((1-t)x+ty)\geq f(x)^{1-t}g(y)^t \quad \mbox{for all }(x, y)\in \Omega_0\times \Omega_1\, .
	\end{equation*}
	Recall that we extend the function $u_{q, \Omega_i}$, $i=0, 1, t$, by zero to $\R^d\setminus\Omega_i$. Consequently, $f, g, h$ are functions on $\R^d$ and we have
	\begin{equation*}
		h((1-t)x+ty)\geq f(x)^{1-t} g(y)^t \quad \mbox{for all }(x, y)\in \R^d\times \R^d\, .
	\end{equation*}

	We now use the Pr\'ekopa--Leindler inequality and the characterization of its cases of equality; see \cite[Theorem 12]{Dubuc} and also \cite[Theorem 21]{colesanti_brunnminkowski_2005}. By this inequality, we deduce that
	\begin{align*}
		\lambda_q(\Omega_t)^{\frac{2}{q-2}} =\int_{\R^d}h(x)\, dx &\geq
		\biggl(\int_{\R^d}f(x)\, dx\biggr)^{1-t}\biggl(\int_{\R^d} g(x)\, dx\biggr)^t = \lambda_q(\Omega_0)^{(1-t)\frac{2}{q-2}}\lambda_q(\Omega_1)^{t\frac{2}{q-2}}\, .
	\end{align*}
	Since, by assumption, we have equality here, we deduce from the characterization of equality in the Pr\'ekopa--Leindler inequality that there is a log-concave function $F$, as well as $\kappa, \eta>0$ and $x'\in \R^d$, so that 
	$$
	f(x) = F(x) 
	\qquad\text{and}\qquad
	g(x)=\kappa\, F(\eta x+x')
	\qquad\text{for almost every}\ x\in\R^d \, .
	$$
	
	We set $U:=\{F>0\}$ and note that, by log-concavity of $F$, $U$ is convex.
	
	We claim that
	\begin{equation}\label{eq:ff}
		\Omega_0 \subset U
		\qquad\text{and}\qquad
		|U\setminus \Omega_0 | = 0 \, .
	\end{equation}
	For the proof we will use the fact, shown in Lemma \ref{ellipticreg}, that $\Omega_0=\{f>0\}$. Thus, if $x\in U\setminus\Omega_0$, then $F(x)>0=f(x)$ and therefore such $x$ belong to the zero measure set, where $F$ and $f$ do not coincide. This proves the second assertion in \eqref{eq:ff}. To prove the first one, we argue by contradiction and assume that there is an $x_0\in\Omega_0$ with $F(x_0)=0$. By Hahn--Banach, there is an affine hyperplane passing through $x_0$ such that $U$ lies on one side of it. Thus, there is an affine halfspace $H$, containing $x_0$ on its boundary, where $F$ vanishes. Since $\Omega_0$ is open, the intersection $\Omega_0\cap H$ has positive measure and for all $x\in \Omega_0\cap H$ one has $0=F(x)<f(x)$. This is a contraction, and the proof of \eqref{eq:ff} is complete.
	
	Noting that $\{ F(\eta\, \cdot +x')>0\} = \eta^{-1}(U-x')$, we obtain, by the same argument, 
	\begin{equation}\label{eq:ff2}
		\Omega_1 \subset \eta^{-1}(U-x')
		\qquad\text{and}\qquad
		|\eta^{-1}(U-x') \setminus \Omega_1 | = 0 \, .
	\end{equation}
	Properties \eqref{eq:ff} and \eqref{eq:ff2} imply that, up to sets of measure zero, $\Omega_0$ and $\Omega_1$ are homothetic copies of the set $U$.

	Since $\Omega_0$ and $\Omega_1$ are open and since the Minkowski sum is not affected by sets of measure zero being removed from $U$, we deduce that $\Omega_t = (1-t)\Omega_0+t\Omega_1= (1-t)U + t\eta^{-1}(U-x')$. In particular, the Brunn--Minkowski-type inequality applied with $\Omega_0=U$ and $\Omega_1= \eta^{-1}(U-x')$ and the monotonicity of $\lambda_q$ under inclusions implies that
	\begin{equation*}
		\lambda_q(\Omega_t)^{-\alpha_q} \geq (1-t)\lambda_q(U)^{-\alpha_q}+t\lambda_q(\eta^{-1}(U-x'))^{-\alpha_q}\geq (1-t)\lambda_q(\Omega_0)^{-\alpha_q}+t\lambda_q(\Omega_1)^{-\alpha_q}\, .
	\end{equation*}
	Since the left-hand side is equal to the right-hand side by assumption, it must hold that $\lambda_q(\Omega_0)=\lambda_q(U)$ and $\lambda_q(\Omega_1)= \lambda_q(\eta^{-1}(U-x'))$. In particular, $u_{q, \Omega_0}$ is a minimizer for $\lambda_q(U)$, so by uniqueness of minimizers and continuity (Lemma~\ref{ellipticreg}) $u_{q, \Omega_0}=u_{q, U}$ in~$U$. If $U\setminus \Omega_0$ had positive capacity, then $u_{q, \Omega_0}$ would have to vanish on this set. But, by Lemma~\ref{ellipticreg}, $u_{q, U}$ is nowhere vanishing in $U$, and so it follows that $\Omega_0$ agrees with $U$ up to a set of capacity zero. The same argument implies that the analogue statement for $\Omega_1$.
\end{proof}

In the first inclusion in \eqref{eq:ff2} one cannot expect equality in general. For instance, if $\Omega_0$ is an open ball in $\R^d$, $d\geq 2$, with a point removed, then $u_{q, \Omega_0}$ extends continuously to this point and assumes there a positive value, while we have agreed to extend $u_{q, \Omega_0}$ by zero to the complement of $\Omega_0$.


\end{document}